\documentclass[11pt,reqno]{amsart}

\usepackage{hyperref,url}

 \usepackage{float}
 \usepackage{mathtools}
\usepackage{amssymb}
\usepackage{bbm}
\usepackage{amssymb}
\usepackage{amsmath, amsfonts}
\usepackage{amscd}
\usepackage{amsthm}
\usepackage{setspace}
\usepackage{enumerate}
\usepackage{graphicx}
\usepackage{mathtools}
\usepackage{tcolorbox} 
\usepackage{subfigure} 
\usepackage{siunitx}
\usepackage{tikz-cd}
\usepackage{bm}
\numberwithin{equation}{section}

\usepackage{mathtools}
\usepackage[tableposition=top]{caption}
\usepackage{booktabs,dcolumn}

\newcommand\dist{{\operatorname{dist}}}

\newcommand\eps{\varepsilon}

\DeclareMathOperator{\im}{im}

\DeclareMathOperator{\LV}{V}

\DeclareMathOperator{\dom}{dom}
\DeclareMathOperator{\Id}{Id}
\DeclareMathOperator{\mesh}{mesh}

\newcommand{\quash}[1]{}

\def\XXint#1#2#3{{\setbox0=\hbox{$#1{#2#3}{\int}$}
     \vcenter{\hbox{$#2#3$}}\kern-.5\wd0}}

\title[C\`{a}dl\`{a}g modifications of Markov Processes]{C\`{a}dl\`{a}g modifications of Markov Processes}
\author{Roni Edwin}
\newtheorem{theorem}{Theorem}[section]
\newtheorem{definition}[theorem]{Definition}

\newtheorem{lemma}[theorem]{Lemma}

\newtheorem{proposition}[theorem]{Proposition}

\begin{document}

\begin{abstract}
     This is a report of the work done in the David Harold Blackwell Summer Research Institute (DHBSRI). Here we give a proof of the existence of c\`{a}dl\`{a}g modification of Markov Processes (on an appropriate space) with Feller semigroup.
\end{abstract}

\maketitle

\tableofcontents

\renewcommand{\theenumi}{\alph{enumi}}
\section{Introduction}
The main focus of our project at DHBSRI was proving the existence c\`{a}dl\`{a}g modifications of Markov processes on a finite state space. Let us unpack these definitions. There is the familiar definition of a Markov process as a sequence of Random Variables $\left(X_n\right)_{n \in \mathbb{N}}$ taking values in a finite state space $S=\left\{s_1, . . ., s_m\right\}$ such that the transition probabilities \begin{equation*}
    \left(\mathbb{P}\left[X_{k+1}=s_j \ | \ X_k=s_i\right]\right)_{1\le i,j\le n}
\end{equation*} are independent of the `time' parameter $k$. 

In the context of Probability theory, one typically works with a more general definition. We first give a brief review of some relevant concepts from Probability theory. We assume some basic familiarity with concepts from measure theory (the definition of a $\sigma$-algebra, a measure/measurable space, etc). 
\subsection{Review}
Let $\Omega$ be a set. Given a collection of subsets $\mathcal{A}$ of $\Omega$, we denote by $\sigma(\mathcal{A})$ the smallest $\sigma$-algebra containing $\mathcal{A}$. So \begin{align*}
    \sigma(\mathcal{A})\coloneqq \bigcap_{\substack{\Sigma\supset \mathcal{A} \\ \Sigma\text{ is a $\sigma$-algebra}}}\Sigma.
\end{align*} $\sigma(\mathcal{A})$ is also referred to as the $\sigma$-algebra generated by $\mathcal{A}$. To wit, for a topological space $S$, we let $\mathcal{B}(S)$ be the Borel $\sigma$-algebra on $S$ (the $\sigma$-algebra generated by the open sets in $S$). Let $(\Omega,\Sigma)$, $\left(\Omega',\Sigma'\right)$ be measure spaces, and let $\left(X_i\right)_{i \in I}$ is a collection of functions from $\Omega$ to $\Omega'$. The $\sigma$-algebra generated by the functions $\left(X_i\right)_{i \in I}$, written  $\sigma\mathopen{}\left(\left\{X_i:i \in I\right\}\right)$, is the $\sigma$-algebra generated by the preimages of the functions $X_i$ on the elements of $\Sigma'$, so \begin{align*}
    \sigma\mathopen{}\left(\left\{X_i:i \in I\right\}\right)\coloneqq \sigma\mathopen{}\left(\left\{X_i^{-1}(U): U \in \Sigma',i \in I\right\}\right).
\end{align*} Equivalently, it is the smallest $\sigma$-algebra with respect to which all the functions $\left(X_i : i\in I\right)$ are measurable. With this, we can talk about product $\sigma$-algebras.
\begin{definition}[Product $\sigma$-algebras]
    Let $\mathcal{A}$ be an arbitrary index set, and for each $\alpha \in \mathcal{A}$, let $\left(\Omega_\alpha,\Sigma_\alpha\right)$ be a measurable space. The Cartesian product space $\Omega\coloneqq \prod_{\alpha \in \mathcal{A}}\Omega_\alpha$ is the space of all functions $\omega\colon \mathcal{A} \to \bigcup_{\alpha \in \mathcal{A}}\Omega_\alpha$ such that for each $\alpha \in \mathcal{A}$, $\omega(\alpha) \in \Omega_\alpha$. Coordinate projection
maps $\left\{\pi_\alpha\colon \Omega \to \Omega_\alpha:\alpha \in \mathcal{A}\right\}$ on $\Omega$ are defined $\pi_\alpha(\omega)=\omega(\alpha)$. With this, the product $\sigma$-algebra $\otimes_{\alpha \in \mathcal{A}}\Sigma_\alpha$ is the $\sigma$-algebra generated by the coordinate projections $\left\{\pi_\alpha:\alpha \in \mathcal{A}\right\}$.
\end{definition}
Note that the product $\sigma$-algebra  $\otimes_{\alpha \in \mathcal{A}}\Sigma_\alpha$ is in general \emph{not} the $\sigma$-algebra generated by the collection of Cartesian products of sets from the respect $\sigma$-algebras $\Sigma_\alpha$. We now define a filtration.

\begin{definition}[Filtration]
    A Filtration on a set $\Omega$ is a collection of $\sigma$-algebras  $\left(\mathcal{F}_t\right)_{t\ge 0}$ on $\Omega$ such that for all $s,t\ge 0$, $s\le t$ implies $\mathcal{F}_s \subset \mathcal{F}_t$.
\end{definition}
Often times, included in the filtration $\left(\mathcal{F}_t\right)_{t\ge 0}$ is a larger $\sigma$-algebra $\mathcal{F}_\infty$, satisfying $\mathcal{F}_t \subset \mathcal{F}_\infty$ for all $t\ge 0$.
We say a stochastic process $\left(X_t\right)_{t\ge 0}$ is adapted to a filtration $\left(\mathcal{F}_t\right)_{t\ge 0}$ if for each $t\ge 0$, the random variable $X_t$ is $\mathcal{F}_t$-measurable. 
\begin{definition}[Conditional expectation]
    Let $\left(\Omega,\Sigma,\mathbb{P}\right)$ be a probability space, and let $X\colon \Omega \to \mathbb{R}$ be an integrable random variable. Let $\mathcal{G}\subset \Sigma$ be a $\sigma$-algebra. The conditional expectation of $X$ given $\mathcal{G}$ is the unique (up to sets of measure $0$) $\mathcal{G}$-measurable random variable $Z\colon \Omega \to \mathbb{R}$ such that \begin{align*}
        \mathbb{E}[Xg]=\mathbb{E}[Zg]
    \end{align*} for each $\mathcal{G}$-measurable function $g\colon \Omega \to \mathbb{R}$. We write $Z=\mathbb{E}[X|\mathcal{G}]$.
\end{definition} We take it for granted that $\mathbb{E}[Z|\mathcal{G}]$ exists and is unique (up to sets of measure $0$). If $Y\colon \Omega \to \mathbb{R}$ is a random variable, we write $\mathbb{E}[X|Y]$ to mean the conditional expectation of $X$, given $\sigma(Y)$, so $\mathbb{E}[X|Y]\coloneqq \mathbb{E}[X|\sigma(Y)]$. An important property of conditional expectation is that \begin{align*}
    \mathbb{E}\left[\mathbb{E}[X|\mathcal{G}]\right]=\mathbb{E}[X],
\end{align*} for any integrable random variable $X\colon \Omega \to \mathbb{R}$ and sub $\sigma$-algebra $\mathcal{G}$.

With these concepts, we can now give a general definition of a Markov Process, that coincides with that given in Section 2.3 of \cite{seppalainen2012basics}.
\begin{definition}[Markov Process]
Let $(E,\mathcal{E})$ be a topological space equipped with its Borel $\sigma$-algebra. A stochastic process $\left(X_t\right)_{t\ge 0}$ on a probability space $\left(\Omega, \Sigma,\mathbb{P}\right)$ with values in $E$, adapted to a filtration $\left(\mathcal{F}_t\right)_{t\ge 0}$, is a Markov Process if for all bounded measurable functions $f:E \to \mathbb{R}$ and $s,t\ge 0$ one has \begin{equation*}
    \mathbb{E}\left[f\mathopen{}\left(X_{s+t}\right)\mathclose{}  |  X_s\right]=\mathbb{E}\left[f\mathopen{}\left(X_{s+t}\right)\mathclose{}  |  \mathcal{F}_s\right].
\end{equation*}  We refer to the equality above as the \emph{Markov Property}.
\label{firstdefMarkovprocess}
\end{definition}
Loosely speaking, this means the transition probabilities at a particular state only depend on the information at that time. We now introduce the idea of a \emph{transition semigroup}, following the definition given in \cite{LeGall}.
Let $E$ be a metrizable locally compact topological space. We also assume that $E$ is $\sigma$-compact, meaning that $E$ is a countable union
of compact sets. The space $E$ is equipped with its Borel $\sigma$-algebra $\mathcal{E}$. In this case, one can find an increasing sequence $\left\{K_n\right\}_{n=1}^\infty$ of compact subsets of $E$,
such that any compact set of $E$ is contained in $K_n$ for some $n$. A function $f\colon E \to \mathbb{R}$
tends to $0$ at infinity if, for every $\eps>0$, there exists a compact subset $K$ of $E$ such
that $|f(x)|\le \eps$ for all $x \in E\setminus K$. This is equivalent to requiring that
\begin{equation*}
    \sup_{x \in E\setminus K_n}|f(x)| \to 0
\end{equation*} as $n \to \infty$.
We let $C_0(E)$ stand for the set of all continuous real functions on E that tend to 0
at infinity, and $C(E)$ the space of all bounded continuous functions on $E$. The spaces $C_0(E)$ and $C(E)$ are a Banach spaces with the supremum norm.

Subsequently, unless stated otherwise, $(E,\mathcal{E})$ denotes a metrizable locally compact topological space that is $\sigma$-compact, equipped with its Borel $\sigma$-algebra $\mathcal{E}$.

A \emph{transition kernel} from $E$ into $E$ is a mapping $Q\colon E\times \mathcal{E} \to [0,1]$ satisfying the following two properties:
\begin{enumerate}
    \item For every $x \in E$, the mapping $\mathcal{E}\ni A \mapsto  Q(x,A)$ is a probability measure on $(E,\mathcal{E})$.
    \item For every $A \in \mathcal{E}$, the mapping $E\ni x \mapsto Q(x,A)$ is $\mathcal{E}$-measurable.
\end{enumerate} 
If $f\colon E \to \mathbb{R}$ is bounded and measurable, or non-negative and measurable, we denote by $Qf$ the function defined by \begin{equation}
    Qf(x)=\int_EQ(x,\textup{d}y)f(y).
    \label{Qfdef}
\end{equation} This allows us to define a transition semigroup on $E$. \begin{definition}[Definition $6.1$ in \cite{LeGall}]
    A collection $\left(Q_t\right)_{t\ge 0}$ of transition kernels on $E$ is called a transition semigroup if the following $3$ properties hold:
    \begin{enumerate}
        \item For every $x \in E$, $Q_0(x,\textup{d}y)=\delta_x(\textup{d}y)$.
        \item For every $s,t\ge 0$ and $A \in \mathcal{E}$, \begin{equation*}
            Q_{s+t}(x,A)=\int_EQ_t(x,\textup{d}y)Q_s(y,A).
        \end{equation*} Equivalently, interpreted as maps from $L^\infty(E)$ to $L^\infty(E)$ via the definition in \eqref{Qfdef}, we have \begin{align*}
            Q_{s+t}=Q_sQ_t.
        \end{align*}
        \item For every $A \in \mathcal{E}$, the function $(t,x) \mapsto Q_t(x,A)$ is measurable with respect to the product $\sigma$-algebra $\mathcal{B}([0,\infty))\otimes \mathcal{E}$.
    \end{enumerate}
\end{definition}
With this, we can give a more specific definition of a Markov process .
\begin{definition}[Time-homogeneous Markov process with respect to semigroup]
Let $\left(Q_t\right)_{t\ge 0}$ be a transition semigroup on $E$.
A Markov process $\left(X_t\right)_{t\ge 0}$ adapted to a filtration $\left(\mathcal{F}_t\right)_{t\ge 0}$, taking values in $E$, per Definition \ref{firstdefMarkovprocess}, with transition semigroup $\left(Q_t\right)_{t\ge 0}$ is one such that, for every $s,t\ge 0$ and bounded measurable function $f\colon E \to \mathbb{R}$, we have \begin{align*}
    \mathbb{E}\left[f(X_{s+t})  | \mathcal{F}_s\right]=Q_tf(X_s).
\end{align*}
\end{definition}
Here the phrase `time-homogeneous' refers to the fact that the transition probabilities from $X_s$ to $X_{s+t}$ depends only on $t$.

Let $\gamma$ be the distribution of $X_0$. Observe that as a consequence of this definition that for any $\varphi \in C(E^k)$ and reals $0\le t_1\le t_2\le \cdots\le t_k$, we have \begin{equation}
\begin{split}
    &\mathbb{E}\left[\varphi\mathopen{}\left(B_{t_1}, . . . , B_{t_k}\right)\mathclose{}\right] \\
    &=\int_E\gamma(\textup{d}x_0)\int_EQ_{t_1}(x_0,\textup{d}x_1)\int_EQ_{t_2-t_1}(x_1,\textup{d}x_2)\cdots\int_EQ_{t_k-t_{k-1}}(x_{k-1},\textup{d}x_k)\varphi(x_1, . . . , x_k).
\end{split}
    \label{porcelain}
\end{equation} This can be proven by induction on $k$:
\begin{proof}
By linearity it suffices to prove this when $\varphi$ is of the form \begin{align*}
    \varphi(x_1, . . . , x_k)=\prod_{j=1}^k\varphi_j(x_j).
\end{align*}
    The Markov Property implies
    $\mathbb{E}[\varphi(B_s)|B_0]=Q_s\varphi(B_0)$, and taking the expectation of both sides, we get \begin{align*}
    \mathbb{E}[\varphi(B_s)]=\mathbb{E}[Q_s\varphi(B_0)]=\int_\Omega Q_s\varphi(B_0(\omega))\mathbb{P}(\textup{d}\omega).
\end{align*} If $\gamma$ denotes the distribution of $B_0$, this then becomes \begin{align*}
    \mathbb{E}[\varphi(B_s)]&=\int_E Q_s\varphi(x_0)\gamma(\textup{d}x_0)=\int_E\gamma(\textup{d}x_0)\int_EQ_s(x_0,\textup{d}x_1)\varphi(x_1).
\end{align*} For the general case, suppose \eqref{porcelain} holds for all choices of $k$ non-negative reals, $k\le p$. The Markov Property implies
\begin{equation*}
    \mathbb{E}\left[\varphi_{p+1}\mathopen{}\left(B_{t_{p+1}}\right)|\mathcal{F}_{t_p}\right]=Q_{t_{p+1}-t_p}\varphi_{p+1}\mathopen{}\left(B_{t_p}\right),
\end{equation*}
and since the random variable $\varphi_1(B_{t_1})\cdots \varphi_p(B_{t_p})$ is $\mathcal{F}_{t_p}$-measurable, from the definition of conditional expectation we may deduce \begin{align*}
    \mathbb{E}\left[\varphi_1\mathopen{}\left(B_{t_1}\right)\mathclose{}\cdots \varphi_p\mathopen{}\left(B_{t_p}\right)\mathclose{}\varphi_{p+1}(B_{t_{p+1}})\right]=\mathbb{E}\left[\varphi_1\mathopen{}\left(B_{t_1}\right)\mathclose{}\cdots \varphi_p\mathopen{}\left(B_{t_p}\right)\mathclose{}Q_{t_{p+1}-t_p}\varphi_{p+1}\mathopen{}\left(B_{t_p}\right)\mathclose{}\right].
\end{align*} Combining this with the inductive assumption, we get \begin{align*}
    &\mathbb{E}\left[\varphi_1\mathopen{}\left(B_{t_1}\right)\mathclose{}\cdots \varphi_p\mathopen{}\left(B_{t_p}\right)\mathclose{}\varphi_{p+1}(B_{t_{p+1}})\right] \\
    &=\int_E\gamma(\textup{d}x_0)\int_EQ_{t_1}(x_0,\textup{d}x_1)\int_EQ_{t_2-t_1}(x_1,\textup{d}x_2)\cdots \int_EQ_{t_p-t_{p-1}}(x_{p-1},\textup{d}x_p)\times \\
    &\varphi_1(x_1)\cdots\varphi_p(x_p)Q_{t_{p+1}-t_p}\varphi_{p+1}(x_p) \\
    &=\int_E\gamma(\textup{d}x_0)\int_EQ_{t_1}(x_0,\textup{d}x_1)\int_EQ_{t_2-t_1}(x_1,\textup{d}x_2)\cdots \int_EQ_{t_p-t_{p-1}}(x_{p-1},\textup{d}x_p)\times \\
    &\varphi_1(x_1)\cdots\varphi_p(x_p)\int_EQ_{t_{p+1}-t_p}(x_p,\textup{d}x_{p+1})\varphi_{p+1}(x_{p+1}) \\
    &=\int_E\gamma(\textup{d}x_0)\int_EQ_{t_1}(x_0,\textup{d}x_1)\int_EQ_{t_2-t_1}(x_1,\textup{d}x_2)\cdots \int_EQ_{t_{p+1}-t_p}(x_p,\textup{d}x_{p+1})\varphi_1(x_1)\cdots \varphi_{p+1}(x_{p+1}),
\end{align*} as desired.
\end{proof}

It turns out the converse of \eqref{porcelain} is also true, in the following sense:
\begin{theorem}
    Suppose $E$ is a Polish space ($E$ is separable, metrizable, and complete with respect to the topology-inducing metric), and $\left(Q_t\right)_{t\ge 0}$ is a transition semigroup on $(E,\mathcal{E})$. Let $E^{[0,\infty)}=\prod_{t \in [0,\infty)}E$ denote the product space, and let $\left(B_t\right)_{t\ge 0}$ denote the canonical process on $E^{[0,\infty)}$, given by $B_t(\omega)=\omega(t)$. Given a probability measure $\gamma$ on $E$, there exists a unique probability measure $\mathbb{P}$ on $\left(E^{[0,\infty)},\otimes_{t \in [0,\infty)}\mathcal{E}\right)$ such that for all continuous functions $\varphi\colon E^k \to \mathbb{R}$ and $k$ reals $0\le t_1\le t_2\le \cdots\le t_k$, \eqref{porcelain} holds.
Moreover, $\left(B_t\right)_{t\ge 0}$ is a Markov process adapted to the filtration $\left(\sigma\mathopen{}\left(\left\{B_\tau:0\le \tau\le t\right\}\right)\right)_{t\ge 0}$, with semigroup $\left(Q_t\right)_{t\ge 0}$.
\end{theorem}
This can be proven by invoking Kolmogorov's extension theorem.

Given a Markov process $\left(X_t\right)_{t\ge 0}$ on $\left(\Omega,\Sigma,\mathbb{P}\right)$ taking values in $E$ with transition kernel $\left(Q_t\right)_{t\ge 0}$,
of particular importance is the regularity of the sample paths $t \mapsto X_t(\omega)$, for fixed $\omega \in \Omega$. Often times, such paths may not be particularly regular, say continuous, or possess the weaker property of being c\`{a}dl\`{a}g (right-continuous with left-limits). Sometimes we can modify $\left(X_t\right)_{t\ge 0}$ to obtain a new process $\left(\widetilde{X}_t\right)_{t\ge 0}$ (being a modification means for each $t>0$, $X_t=\widetilde{X}_t$ almost surely) which is more regular than the original. To that end, we introduce the idea of a c\`{a}dl\`{a}g process:
\begin{definition}
    A Stochastic process $\left(X_t\right)_{t\ge 0}$ on a probability space $\left(\Omega,\Sigma,\mathbb{P}\right)$ taking values in $E$ is called c\`{a}dl\`{a}g if for every $\omega \in \Omega$, the sample path $t \mapsto X_t(\omega)$ is c\`{a}dl\`{a}g, so it is right-continuous with left-limits.
\end{definition}
An important theorem in the theory of Markov Processes asserts that under some conditions, one can obtain a c\`{a}dl\`{a}g modification of a given Markov Process. To understand when this is possible, we start by introducing the idea of a Feller semigroup. There are two slightly different definitions common in the literature:
\begin{definition}
    Let $\left(Q_t\right)_{t\ge 0}$ be a transition semigroup on $E$. We say that $\left(Q_t\right)_{t\ge 0}$ is a \emph{Feller semigroup} if:
    \begin{enumerate}
        \item For all $f \in C_0(E)$, $Q_tf \in C_0(E)$, and 
        \item For all $f \in C_0(E)$, $\lVert Q_tf-f\rVert_{C_0(E)} \to 0$ as $t \to 0$.
    \end{enumerate}
    \label{Fellersemigroupdef1}
\end{definition}
Some authors only require that $Q_t$ maps $C(E) \to C(E)$, hence the following alternative definition:
\begin{definition}
Let $\left(Q_t\right)_{t\ge 0}$ be a transition semigroup on $E$. We say that $\left(Q_t\right)_{t\ge 0}$ is a \emph{Feller semigroup} if:
    \begin{enumerate}
        \item For all $f \in C(E)$, $Q_tf \in C(E)$, and 
        \item For all $f \in C(E)$, $\lVert Q_tf-f\rVert_{C(E)} \to 0$ as $t \to 0$.
    \end{enumerate}
    \label{Fellersemigroupdef2}
\end{definition}
With this, the theorem referenced above is as follows.
\begin{theorem}[Theorem 6.15 in \cite{LeGall}]
Let $\left(X_t\right)_{t\ge 0}$ be a Markov process with Feller semigroup $\left(Q_t\right)_{t\ge 0}$ (according to Definition \ref{Fellersemigroupdef1}), adapted to the Filtration $\left(\mathcal{F}_t\right)_{t \in [0,\infty]}$. Set $\widetilde{\mathcal{F}}_\infty=\mathcal{F}_\infty$, and for every $t\ge 0$, set \begin{equation}
    \widetilde{\mathcal{F}}_t=\sigma\left( \mathcal{N}\cup \bigcap_{\substack{s>t}}\mathcal{F}_s\right),
    \label{drizzy}
\end{equation} where $\mathcal{N}$ is the class of all $\mathcal{F}_\infty$-measurable sets with $0$ probability.
Then the process $\left(X_t\right)_{t\ge 0}$ has a c\`{a}dl\`{a}g modification $\left(\widetilde{X}_t\right)_{t\ge 0}$ which is adapted to the Filtration $\left(\widetilde{\mathcal{F}}_t\right)_{t\ge 0}$. Moreover, $\left(\widetilde{X}_t\right)_{t\ge 0}$ is a Markov Process with semigroup $\left(Q_t\right)_{t\ge 0}$.
\label{pesos}
\end{theorem}
The main focus of our project was proving a specific case of this theorem (for example, when the space $E$ is finite). We considered the following special case:
\begin{proposition}
    Let $\left(\Omega,\Sigma, \mathbb{P}\right)$ be the underlying probability space. Let $\left(B_t\right)_{t\ge 0}$ be a Markov process on $\Omega$ with values in $E$, adapted to the Filtration $\left(\mathcal{F}_t\right)_{t\in [0,\infty]}$, with Feller semigroup $\left(Q_t\right)_{t\ge 0}$ (according to Definition \ref{Fellersemigroupdef2}). Suppose additionally that $t \mapsto Q_t$ is continuous with respect to the operator norm topology on the space of bounded linear operators on $C(E)$. Set \begin{align*}
    \mathcal{F}_t^+\coloneqq \bigcap_{\substack{ s>t}}\mathcal{F}_s,
\end{align*} and $\widetilde{\mathcal{F}}_t=\sigma\mathopen{}\left(\mathcal{F}_t^+\cup \mathcal{N}\right)\mathclose{}$, where $\mathcal{N}$ is the class of all $\mathcal{F}_\infty$-measurable sets with zero probability. 
Then the process $\left(B_t\right)_{t\ge 0}$ has a c\`{a}dl\`{a}g modification $\left(\widetilde{B}_t\right)_{t\ge 0}$ which is adapted to the Filtration $\left(\widetilde{\mathcal{F}}_t\right)_{t\ge 0}$. Moreover, $\left(\widetilde{B}_t\right)_{t\ge 0}$ is a Markov Process with semigroup $\left(Q_t\right)_{t\ge 0}$, adapted to the filtration $\left(\widetilde{\mathcal{F}}_t\right)_{t\ge 0}$.
\end{proposition}
We note the condition that $t \mapsto Q_t$ is continuous with respect to the operator norm on the space of bounded linear operators on $C(E)$  is superfluous when $E$ is finite (if $\left(Q_t\right)_{t\ge 0}$ is a Feller semigroup on a finite space $E$, then it is necessarily continuous in the operator norm topology), which was the initial focus of our project.

\section{Preliminaries}
We start by introducing a lemma that allows us to characterise $\sigma$-algebras of the form \eqref{drizzy}:
\begin{lemma}
Let $\Sigma \subset \mathcal{F}_\infty$ be a $\sigma$-algebra on $\Omega$, and let $\mathcal{N}$ denote the collection of $\mathcal{F}_\infty$-measurable sets with $0$ probability. Then
\begin{align*}
    \sigma(\Sigma\cup \mathcal{N})= \left\{G \in \mathcal{F}_\infty:\exists F \in \Sigma\text{ such that }\mathbb{P}(F\cap G^\mathsf{c})=\mathbb{P}(F^\mathsf{c}\cap G)=0\right\},
\end{align*} where $G^\mathsf{c}$ denotes the complement of $G$ in $\Omega$. 
    \label{coachella}
\end{lemma}
\begin{proof}
Let \begin{align}
        \mathfrak{S}\coloneqq \left\{G \in \mathcal{F}_\infty:\exists F \in \Sigma\text{ such that }\mathbb{P}(F\cap G^\mathsf{c})=\mathbb{P}(F^\mathsf{c}\cap G)=0\right\}.
        \label{mathcalXdef}
    \end{align}
    First note that $\Sigma \subset \mathfrak{S}$, since for any $F \in \Sigma$, we have $F\cap F^\mathsf{c}=\emptyset$. Similarly, for any $N \in \mathcal{N}$, we can take $\emptyset \in \Sigma$, in which case, $\mathbb{P}\bigl(\emptyset \cap N^\mathsf{c}\bigr)=\mathbb{P}\bigl(\Omega\cap N\bigr)=0$. Hence $\mathcal{N}\subset \mathfrak{S}$. So \begin{equation}
    \Sigma\cup \mathcal{N}\subset \mathfrak{S}.
        \label{alacarte}
    \end{equation}
    We now show that $\mathfrak{S}$ is a $\sigma$-algebra. We first show that $\mathfrak{S}$ is closed under complements. Consider any $G \in \mathfrak{S}$, and let $F \in \Sigma$ be such that $\mathbb{P}(F\cap G^\mathsf{c})=\mathbb{P}(F^\mathsf{c}\cap G)=0$. Then \begin{align*}
        &\mathbb{P}\mathopen{}\left(F^\mathsf{c}\cap \bigl(G^\mathsf{c}\bigr)^{\mathsf{c}}\right)\mathclose{}=\mathbb{P}\bigl(F^\mathsf{c}\cap G\bigr)=0,\quad\text{ and } \\
&\mathbb{P}\mathopen{}\left(\bigl(F^\mathsf{c}\bigr)^\mathsf{c}\cap G^\mathsf{c}\right)\mathopen{}=\mathbb{P}\bigl(F\cap G^\mathsf{c}\bigr)=0.
    \end{align*} So this means $G^\mathsf{c} \in \mathfrak{S}$ if $G \in \mathfrak{S}$. Now, we just have to show that $\mathfrak{S}$ is closed under countable unions. To that end, let $I$ be a countable index set, and let $\left\{G_i:i \in I\right\}$ be a countable collection of elements of $\mathfrak{S}$. For each $i \in I$, let $F_i \in \Sigma$ be such that \begin{align*}
        \mathbb{P}\bigl(F_i\cap G_i^\mathsf{c}\bigr)=\mathbb{P}\bigl(F_i^\mathsf{c}\cap G_i\bigr)=0.
    \end{align*} Then $\bigcup_{j \in I}F_j \in \Sigma$ since $\Sigma$ is a $\sigma$-algebra, and \begin{align*}
        \left(\bigcup_{i \in I}G_i\right)\cap \left(\bigcup_{j \in I}F_j\right)^\mathsf{c}&=\bigcup_{i \in G_i}\left(G_i\cap \bigcap_{j \in I}F_j^\mathsf{c}\right)\subset \bigcup_{i \in G_i}\left(G_i\cap F_i^\mathsf{c}\right).
    \end{align*} Since $\mathbb{P}\bigl(G_i\cap F_i^\mathsf{c}\bigr)=0$ for each $i$, this implies \begin{align*}
        \mathbb{P}\mathopen{}\left( \left(\bigcup_{i \in I}G_i\right)\cap \left(\bigcup_{j \in I}F_j\right)^\mathsf{c}\right)\mathclose{}=0.
    \end{align*} A similar argument (just switch $G_i$ with $F_i$ in the argument above) shows \begin{align*}
        \mathbb{P}\mathopen{}\left( \left(\bigcup_{i \in I}G_i\right)^\mathsf{c}\cap \left(\bigcup_{j \in I}F_j\right)\right)\mathclose{}=0.
    \end{align*} So $\bigcup_{i \in I}G_i \in \mathfrak{S}$. This shows $\mathfrak{S}$ is closed under countable unions. Given it is also closed under complements, and \eqref{alacarte} implies $\emptyset, \Omega \subset \mathfrak{S}$, this means $\mathfrak{S}$ is a $\sigma$-algebra. From \eqref{alacarte}, we now know that $\sigma(\Sigma\cup \mathcal{N}) \subset \mathfrak{S}$. To get the reverse inclusion, let $G \in \mathfrak{S}$, and $F \in \Sigma$ be such that \begin{equation*}
        \mathbb{P}(F\cap G^\mathsf{c})=\mathbb{P}(F^\mathsf{c}\cap G)=0.
    \end{equation*} Note that this condition implies $\mathbb{P}\bigl(F\cup G\setminus F\cap G\bigr)=0$, in which case we can write $G$ as $G=\bigl(F\cup N_1\bigr)\setminus N_2$ for sets $N_1,N_2 $ of measure $0$. However, $\bigl(F\cup N_1\bigr)\setminus N_2 \in \sigma(\Sigma\cup \mathcal{N})$, and so we get $G \in \sigma(\Sigma\cup \mathcal{N})$. Hence $\mathfrak{S}\subset \sigma(\Sigma \cup \mathcal{N}$, and so we may deduce\begin{align*}
       \sigma(\Sigma\cup \mathcal{N})=\mathfrak{S}.
    \end{align*} This completes the proof of the lemma.
\end{proof}

\noindent
\newline

The advantage of the stronger condition of $Q_t$ being continuous in the operator norm topology is illustrated in the following lemma:
\begin{lemma}
    Let $\left(Q_t\right)_{t\ge 0}$ be a Feller semigroup (per Definition \ref{Fellersemigroupdef2}) on $(E,\mathcal{E})$. Suppose further that $t \mapsto Q_t$ is continuous in with respect to the operator norm on the space of bounded linear operators on $C(E)$. Then $Q_t=\exp(At)$ for some bounded linear map $A\colon C(E) \to C(E)$.
    \label{nastyc}
\end{lemma}
Let $\mathfrak{B}\mathopen{}\left(C(E)\right)\mathclose{}$ denote the space of bounded linear operators on $C(E)$, equipped with the operator norm. Note that $\mathfrak{B}\mathopen{}\left(C(E)\right)\mathclose{}$ is a Banach space with respect to the operator norm.  First note that since $Q_0=\Id$, the identity map, that for $t$ small enough, $\lVert Q_t-\Id\rVert_{\mathfrak{B}(C(E))}<1$. Consequently, for $t$ small enough, we can define $\log\mathopen{}\left(Q_t\right)\mathclose{}$ via the common taylor series. That is,  define 
    \begin{align*}
        \log \colon \left\{P \in \mathfrak{B}(C(E)): \lVert P-\Id\rVert_{\mathfrak{B}(C(E))}<1\right\} \to \mathfrak{B}(C(E))
    \end{align*}
    by \begin{align}
        \log(P)\coloneqq \sum_{k=1}^\infty \frac{(-1)^{k-1}}{k}(P-\Id)^k.
        \label{oplogdef}
    \end{align} In a similar vein, we define $\exp\colon \mathfrak{B}(C(E)) \to \mathfrak{B}(C(E))$ by
    \begin{align*}
        \exp(P)\coloneqq\sum_{k=0}^\infty \frac{P^k}{k!}.
    \end{align*} 
    Note that $\exp$ as defined is continuous on $\mathfrak{B}(C(E))$.
    We will show that \begin{enumerate}
        \item \label{firststat} $\exp(\log M)=M$ for all $M$ in the domain of $\dom(\log)$, the domain of $\log$, and 
        \item \label{secondstat} There exists $\eps>0$ such that for all $P_1,P_2 \in \mathfrak{B}(C(E))$ that commute with each other, further satisfying $\lVert P_j-\Id\rVert_{\mathfrak{B}(C(E))}<\eps$, $j=1,2$, we have \begin{align*}
            \log\mathopen{}\left(P_1P_2\right)\mathclose{}=\log\mathopen{}\left(P_1\right)\mathclose{}+\log\mathopen{}\left(P_2\right)\mathclose{}.
        \end{align*}
    \end{enumerate}
    We first note that $\log$ as defined is continuous, since for any $r \in (0,1)$, the series defining $\log$ converges uniformly on \begin{align*}
        \left\{P \in \mathfrak{B}(C(E)): \lVert P-\Id\rVert_{\mathfrak{B}(C(E))}\le r\right\}.
    \end{align*}
    \begin{proof}[Proof of \eqref{firststat}]
         We will show that for any $M \in \mathfrak{B}(C(E))$ such that $\lVert M\rVert_{\mathfrak{B}(C(E))}<1$, we have \begin{align*}
        \exp(\log(\Id+M))=\Id+M.
    \end{align*}Let  $M \in \mathfrak{B}(C(E))$ be such that $\lVert M\rVert_{\mathfrak{B}(C(E))}<1$. First note that \begin{align*}
        \sum_{k=1}^n\frac{(-1)^{k-1}M^k}{k} \to \log(\Id+M)
    \end{align*} in $\mathfrak{B}(C(E))$ as $n \to \infty$. Now, consider the sequence of holomorphic functions $\left\{f_n\right\}_{n=1}^\infty$ on the open disk $\mathbb{D}\coloneqq \left\{z \in \mathbb{C}:|z|<1\right\}$ defined by \begin{align*}
        f_n(z)=\exp\left(\sum_{k=1}^n\frac{(-1)^{k-1}z^k}{k}\right),
    \end{align*} so that $f_n(M) \to \exp(\log (\Id+M))$ in $\mathfrak{B}(C(E))$. 
    The idea is to show $f_n(z) \to 1+z$ in a suitable sense. Observe from the Cauchy Derivative formula (CDF), that for any $r \in (0,1)$, we have \begin{align*}
        \frac{f_n^{(m)}(0)}{m!}=\frac1{2\pi i}\int_{|z|=r}\frac{f_n(z)}{z^{m+1}}\textup{d}z.
    \end{align*} Consequently,  \begin{align*}
        \left|\frac{f_n^{(m)}(0)}{m!}\right|&=\frac1{2\pi}\left|\int_{|z|=r}\frac{\exp\left(\sum_{k=1}^n\frac{(-1)^{k-1}z^k}{k}\right)}{z^{m+1}}\textup{d}z\right|\le \frac1{2\pi r^{m}}\int_{0}^{2\pi}\exp\mathopen{}\left(\sum_{k=1}^n\frac{r^{k}}{k}\right)\mathclose{}\textup{d}t\le \frac1{r^m(1-r)}.
    \end{align*} 
     Moreover, for each $m\in \mathbb{N}\cup \{0\}$, from the same CDF, we may deduce \begin{align*}
        \lim_{n \to \infty}\frac{f_n^{(m)}(0)}{m!}=\begin{cases}
            1 & \text{ if }m \in \{0,1\}, \\
            0 & \text{ otherwise. }
        \end{cases}
    \end{align*} Using the holomorphicity of the functions $f_n$, this means for each $n$, we can write \begin{align*}
        f_n(M)=\sum_{k=0}^\infty c(k,n)M^k,
    \end{align*} where \begin{equation*}
        |c(k,n)|\le \frac1{r^k(1-r)}
    \end{equation*} uniformly in $n$, and \begin{align*}
          \lim_{n \to \infty}c(n,k)=\begin{cases}
            1 & \text{ if }k \in \{0,1\}, \\
            0 & \text{ otherwise. }
        \end{cases} 
    \end{align*} Taking $r \in (0,1)$ such that $\lVert M \rVert_{\mathfrak{B}(C(E))}<r$, we may conclude that $\lim_{n \to \infty}f_n(M)=\Id+M$ in $\mathfrak{B}(C(E))$,  so \begin{align*}
        \exp(\log (\Id+M))=\Id+M,
    \end{align*} as desired. 
    \end{proof}
We now prove statement \ref{secondstat}:
\begin{proof}[Proof of statement \ref{secondstat}]
   From statement \ref{firststat} we may deduce $\log$ is an open map, since it implies \begin{align*}
    \lVert M\rVert_{\mathfrak{B}(C(E))}=\lVert \exp(\log M)\rVert_{\mathfrak{B}(C(E))}\le \lVert \exp\rVert_{\mathfrak{B}(C(E))\to \mathfrak{B}(C(E))}\lVert \log M\rVert_{\mathfrak{B}(C(E))}\le e\lVert \log M\rVert_{\mathfrak{B}(C(E))}.
\end{align*} To that end, let $\eps>0$ be small enough so that the open ball of radius $\eps$ is contained in the image of $\log$, so \begin{align}
    \left\{M \in \mathfrak{B}(C(E)): \lVert M\rVert_{\mathfrak{B}(C(E))}<\eps\right\} \subset \im \log.
    \label{jasa}
\end{align} Let $\delta \in (0,1)$ be small, and let $P_1,P_2 \in \mathfrak{B}(C(E))$ be such that $P_1P_2=P_2P_1$, and \begin{align}
    \lVert P_j-\Id \rVert_{\mathfrak{B}(C(E))}<\delta\quad  \forall j \in\{1,2\}.
    \label{mergui}
\end{align} From the definition of $\log$, this implies $\log(P_1)$ commutes with $\log\mathopen{}\left(P_2\right)\mathclose{}$, so \begin{align*}
    \exp\mathopen{}\left(\log\mathopen{}\left(P_1\right)\mathclose{}+\log\mathopen{}\left(P_2\right)\mathclose{}\right)=\exp(\log\mathopen{}\left(P_1\right)\mathclose{})\exp(\log\mathopen{}\left(P_2\right)\mathclose{})=P_1P_2=\exp\mathopen{}\left(\log(P_1P_2)\right)\mathclose{},
\end{align*} ($\exp(A)\exp(B)=\exp(A+B)$ whenever $A$ and $B$ commute) \begin{equation}
\exp\mathopen{}\left(\log(P_1P_2)\right)\mathclose{}= \exp\mathopen{}\left(\log\mathopen{}\left(P_1\right)\mathclose{}+\log\mathopen{}\left(P_2\right)\mathclose{}\right).
\label{webr}
\end{equation}
Now, \eqref{mergui} implies $\lVert \log\mathopen{}\left(P_j\right)\mathclose{}\rVert_{\mathfrak{B}(C(E))}\le -\log(1-\delta)$ for each $j=1,2$, so \begin{align*}
    \lVert \log\mathopen{}\left(P_1\right)\mathclose{}+\log\mathopen{}\left(P_2\right)\mathclose{}\rVert_{\mathfrak{B}(C(E))}<-2\log(1-\delta).
\end{align*} By choosing $\delta$ small enough, we can make $-2\log(1-\delta)<\eps$, which would imply $\log\mathopen{}\left(P_1\right)\mathclose{}+\log\mathopen{}\left(P_2\right)\mathclose{} \in \im \log$, from \eqref{jasa}. In that case, we can take the $\log$ of both sides of \eqref{webr} to get \begin{align*}
    \log(P_1P_2)=\log\mathopen{}\left(P_1\right)\mathclose{}+\log\mathopen{}\left(P_2\right)\mathclose{},
\end{align*} as desired.
\end{proof}
We now prove Lemma \ref{nastyc}.
\begin{proof}
Since $t \mapsto Q_t$ is continuous on $\mathfrak{B}(C(E))$, let $w>0$ be small enough so that \begin{align*}
    \lVert Q_t-\Id\rVert_{\mathfrak{B}(C(E))}<1 \quad \forall t \in [0,w].
\end{align*}
Define $\widehat{Q}\colon [0,w] \to \mathfrak{B}(C(E))$ by \begin{align}
    \widehat{Q}(t)\coloneqq\log\mathopen{}\left(Q_t\right)\mathclose{}.
    \label{exponentiate}
\end{align} Then $\widehat{Q}$ is continuous, and since $Q_s$ commutes with $Q_t$ for all $t,s\ge 0$, the semigroup property $Q_{s+t}=Q_sQ_t$ implies \begin{align*}
    \widehat{Q}(s+t)=\widehat{Q}(s)+\widehat{Q}(t)\quad\forall t,s\in [0,w]\text{ such that }t+s\le w.
\end{align*} With this, we may deduce that $\widehat{Q}(t)=\frac{\widehat{Q}(w)}{w}t$: First we note that \begin{align*}
    \widehat{Q}\mathopen{}\left(\frac{w}{m}\right)\mathclose{}=\frac1{m}\widehat{Q}(w),
\end{align*} for any $m\in \mathbb{N}$.
Pick any integers $m,n\ge 0$ with $0< n\le m$. Then \begin{align*}
    n\widehat{Q}\mathopen{}\left(\frac{w}{m}\right)\mathclose{}=\widehat{Q}\mathopen{}\left(\frac{nw}{m}\right)\mathclose{},
\end{align*} and so \begin{align*}
    \widehat{Q}\mathopen{}\left(\frac{n}{m}\cdot w\right)\mathclose{}=\frac{n}{m}\widehat{Q}(w),
\end{align*} for all integers $0<n\le m$. Consequently,
\begin{align*}
    \widehat{Q}\mathopen{}\left(ws\right)=\widehat{Q}\mathopen{}\left(w\right)\mathclose{}s \quad \forall s \in \mathbb{Q}\cap [0,1].\end{align*} Since $\widehat{Q}$ is continuous, this implies $  \widehat{Q}(t)=\frac{\widehat{Q}(w)}{w}t$ for all $t\in [0,t_0]$. Exponentiating both sides of \eqref{exponentiate}, applying statement \eqref{firststat}, then implies  \begin{align*}
        Q_t=\exp\mathopen{}\left(\frac{\widehat{Q}(w)}{w}t\right)\mathclose{} \quad \forall t \in [0,w].
    \end{align*} We can then use the semigroup property $Q_{s+t}=Q_sQ_t$ to conclude $Q_t=\exp\mathopen{}\left(\frac{\widehat{Q}(w)}{w}t\right)\mathclose{}$ for all $t\ge 0$, as desired.
\end{proof}
Since $E$ is metrizable, let $\rho\colon E\times E \to [0,\infty)$ be a metric that induces the topology on $E$. We introduce a truncated version of $\rho$, $\widetilde{\rho}\colon E\times E \to[0,1]$ given by \begin{equation}
\widetilde{\rho}=\min(1,\rho).
    \label{widehatrhodef}
\end{equation}
Note that $\widetilde{\rho}$ is continuous on $E\times E$.
With this, we have the following theorem.
\begin{theorem}
    For each $T>0$, there is a constant $M_{T}>0$ depending on $T$ such that for all  $0\le s\le t$, such that $t-s\le T$, \begin{align*}
        \mathbb{E}\left[\widetilde{\rho}\mathopen{}\left(B_t,B_s\right)\mathclose{}\right]\le M_T(t-s).
    \end{align*}
    \label{euphoria}
\end{theorem}
\begin{proof}
Note
    Let $\gamma$ be the distribution of $B_0$, so that from the formula in \eqref{porcelain}, \begin{align*}
        \mathbb{E}\left[\widetilde{\rho}\mathopen{}\left(B_t,B_s\right)\mathclose{}\right]=\int_E\gamma(\textup{d}x_0)\int_EQ_s(x_0,\textup{d}x_1)\int_EQ_{t-s}(x_1,\textup{d}x_2)\widetilde{\rho}(x_1,x_2),
    \end{align*} and since $\widetilde{\rho}(x_1,x_1)=0$, we can write this as \begin{align}
\mathbb{E}\left[\widetilde{\rho}\mathopen{}\left(B_t,B_s\right)\mathclose{}\right]=\int_E\gamma(\textup{d}x_0)\int_EQ_s(x_0,\textup{d}x_1)\left(Q_{t-s}\widetilde{\rho}(x_1,\cdot)(x_1)-\widetilde{\rho}(x_1,\cdot)(x_1)\right). 
\label{rosalia}
    \end{align} Here $\widetilde{\rho}(x_1,\cdot)\colon E \to [0,\infty)$ is such that $\widetilde{\rho}(x_1,\cdot)(x)=\widetilde{\rho}(x_1,x)$. From Lemma \ref{nastyc}, we know we can write $Q_t=\exp(At)$ for some $A \in \mathfrak{B}(C(E))$. Consequently, for any function $f \in C(E)$, we can write \begin{align*}
        Q_hf(x)-f(x)=\int_0^hA\exp\mathopen{}\left(At\right)\mathclose{}f(x)\textup{d}t,
    \end{align*} by expanding the series for $\exp$, hence \begin{align*}
        \lVert Q_hf-f\rVert_{C(E)}\le  h\cdot \lVert A \rVert_{\mathfrak{B}(C(E))}\exp\mathopen{}\left(\lVert A \rVert_{\mathfrak{B}(C(E))}h\right)\mathclose{}\lVert f\rVert_{C(E)}.
    \end{align*} This implies for any $x_1 \in E$, \begin{align*}
        \left|Q_{t-s}\widetilde{\rho}(x_1,\cdot)(x_1)-\widetilde{\rho}(x_1,\cdot)(x_1)\right|\le (t-s)\lVert A \rVert_{\mathfrak{B}(C(E))}\exp\mathopen{}\left(\lVert A \rVert_{\mathfrak{B}(C(E))}(t-s)\right)\mathclose{}\lVert \widetilde{\rho} \rVert_{C(E\times E)},
    \end{align*} since $\widetilde{\rho}$ is bounded. Plugging this into \eqref{rosalia}, using the fact that $0\le t-s\le T$, we get \begin{align*}
        \mathbb{E}\left[\widetilde{\rho}(B_t,B_s)\right]&\le (t-s)\int_E\gamma(\textup{d}x_0)\int_E Q_s(x_0,\textup{d}x_1)\lVert A \rVert_{\mathfrak{B}(C(E))}\exp\mathopen{}\left(\lVert A \rVert_{\mathfrak{B}(C(E))}T\right)\mathclose{}\lVert \widetilde{\rho} \rVert_{C(E\times E)}\\
        &\le \lVert A \rVert_{\mathfrak{B}(C(E))}\exp\mathopen{}\left(\lVert A \rVert_{\mathfrak{B}(C(E))}T\right)\mathclose{}\lVert \widetilde{\rho} \rVert_{C(E\times E)}(t-s),
    \end{align*} as desired.
\end{proof}

\noindent
\newline

 For each element $\omega \in \Omega$ and partition  $\pi=\left\{a=\pi_0<\pi_1<\cdots<\pi_{k-1}<\pi_k=b\right\}$ of $[a,b]$, let $\LV\mathopen{}\left(\omega,\pi\right)\mathclose{}$ be the $\widetilde{\rho}$-variation of $\omega$ over $\pi$, given by \begin{align*}
    \LV(\omega,\pi)\coloneqq \sum_{j=1}^k \widetilde{\rho}\mathopen{}\left(B_{\pi_j}\mathopen{}\left(\omega\right),B_{\pi_{j-1}}\mathopen{}\left(\omega\right)\mathclose{}\mathclose{}\right)\mathclose{}.
\end{align*} Note that Theorem \ref{euphoria} implies the following lemma: \begin{lemma}
Let $\tau$ be any partition of the interval $[s,t]$ with $\mesh(\tau)\le 1$. Then there is a constant $K>0$ such that \begin{align*}
    \int_{\Omega}\LV\mathopen{}\left(\omega,\tau\right)\mathbb{P}(\textup{d}\omega)\le K(t-s).
\end{align*}
    \label{goddid}
\end{lemma}
We are now going to define a c\`{a}dl\`{a}g modification $\left(\widetilde{B}_t\right)_{t\ge 0}$ of $\left(B_t\right)_{t\ge 0}$ as follows.\newline

\noindent

\section{Definition of \texorpdfstring{$\left(\widetilde{B}_t\right)_{t\ge 0}$}{Lg}}
To define $\left(\widetilde{B}_{t\ge 0}\right)$ we will need some lemmas. First we introduce a `canonical' sequence refined partitions that cover $\mathbb{Q}$:
\begin{definition}
For each $T \in \mathbb{Q}^+$, let $\left\{\tau^T_{k}\right\}_{k=1}^\infty$ be a sequence of refined rational partitions of $[0,T]$ with $\mesh\mathopen{}\left(\tau_k^T\right)\mathclose{}\le 1$. So for each $k$, $\tau^T_{k} \subset \tau^T_{k+1}$, each element of $\tau^T_k$ is rational, and \begin{align*}
    \bigcup_{k=1}^\infty \tau^T_k=[0,T]\cap \mathbb{Q}.
\end{align*} For example, we can take the elements of $\tau_k^T$ to be \begin{align*}
    \tau_k^T=\left(\{T\}\cup \left\{\frac{p}{q}:p,q\in \mathbb{N}\cup \{0\}, \ \gcd(p,q)=1, \  \text{ and } q\le k\right\}\right)\cap [0,T].
\end{align*}
    \label{canonicalseqpartitionts}
\end{definition}
The next lemma is as follows:
\begin{lemma}
 For each $T\ge 0$ , let $\left\{\tau_k^T\right\}_{k=1}^\infty$ be the partitions defined in Definition \ref{canonicalseqpartitionts}. Set\begin{align}
    \mathcal{S}_T\coloneqq \left\{\omega \in \Omega: \lim_{k \to \infty}\LV\mathopen{}\left(\omega,\tau_k^T\right)\mathclose{}=\infty\right\}.
    \label{mathcalSdef}
\end{align}Note if $T_1<T_2$, then $\mathcal{S}_{T_1}\subset \mathcal{S}_{T_2}$. The set $\mathcal{S}_T$ is $\mathcal{F}_T$-measurable, and $\mathbb{P}\mathopen{}\left(\mathcal{S}_T\right)\mathclose{}=0$.
    \label{canon}
\end{lemma}
\begin{proof}
Since the partitions $\tau_k^T$ are getting finer, the functions $\omega \mapsto \LV\mathopen{}\left(\omega,\tau_k^T\right)\mathclose{}$ form a non-decreasing sequence (in $k$), so the limit exists (or is infinite). So it suffices to show for each partition $\tau_k^T$, the function $\omega \mapsto \LV\mathopen{}\left(\omega,\tau_k^T\right)\mathclose{}$ is $\mathcal{F}_T$-measurable. We can write $\tau_k^T$ as \begin{align*}
        \tau_{k}^T=\left\{0=\tau_{k,0}^T<\tau_{k,1}^T<\cdots<\tau_{k,m-1}^T<\tau_{k,m}^T=T\right\},
    \end{align*} and so \begin{align*}
        \LV\bigl(\omega,\tau_k^T\bigr)=\sum_{j=1}^{m}\widetilde{\rho}\mathopen{}\left(B_{\tau_{k,j}^T}(\omega),B_{\tau_{k,j-1}^T}(\omega)\right)\mathclose{}.
    \end{align*} Since each $\tau_{k,j}^T\le T$ and $\widetilde{\rho}$ is continuous, this shows $\omega\mapsto \LV\mathopen{}\left(\omega,\tau_k^T\right)\mathclose{}$ is $\mathcal{F}_T$-measurable. Consequently, the function $\omega \mapsto \lim_{k \to \infty}\LV\mathopen{}\left(\omega,\tau_k^T\right)\mathclose{}$ is $\mathcal{F}_T$-measurable, and so $\mathcal{S}_T$ as defined is $\mathcal{F}_T$-measurable.

    Since each $\tau_k^T$ is a partition of $[0,T]$ with $\mesh\mathopen{}\left(\tau_k^T\right)\le 1$, from Lemma \ref{goddid}, we have \begin{align*}
        \int_{\Omega}\LV\mathopen{}\left(\omega,\tau_k^T\right)\mathclose{}\mathbb{P}(\textup{d}\omega)\le KT,
    \end{align*} so \begin{align*}
        \int_{\mathcal{S}_T}\LV\mathopen{}\left(\omega,\tau_k^T\right)\mathopen{}\mathbb{P}(\textup{d}\omega)\le KT.
    \end{align*} Combining this with Fatou's lemma, we get \begin{align*}
        KT\ge \liminf_{k \to \infty}\int_{\mathcal{S}_T}\LV\mathopen{}\left(\omega,\tau_k^T\right)\mathopen{}\mathbb{P}(\textup{d}\omega)\ge  \int_{\mathcal{S}_T}\liminf_{k \to \infty}\LV\mathopen{}\left(\omega,\tau_k^T\right)\mathopen{}\mathbb{P}(\textup{d}\omega)=\int_{\mathcal{S}_T}\infty \ \mathbb{P}(\textup{d}\omega),
    \end{align*} so \begin{align*}
        \infty\cdot \mathbb{P}\left(\mathcal{S}_T\right)\le KT,
    \end{align*}which implies $\mathbb{P}\left(\mathcal{S}_T\right)=0$, as desired.
\end{proof}
The second lemma is as follows:
\begin{lemma}
   For each $t\ge 0$, the set \begin{align}
       \mathcal{C}_t&\coloneqq  \left\{\omega \in \Omega: \lim_{\substack{s \to t \\ s \in \mathbb{Q}}}B_s(\omega)=B_t(\omega)\right\}
       \label{mathcalCtdef},
    \end{align} is $\mathcal{F}_t^+$-measurable, and $\mathbb{P}\bigl( \mathcal{C}_t\bigr)=1$.
\label{Yanni}
\end{lemma}
\begin{proof}[Proof of Lemma \ref{Yanni}]
    We start by showing $\mathcal{C}_t$ is $\mathcal{F}_t^+$-measurable. Using the $\eps$-$\delta$ definition of a limit,  \begin{align*}
         \lim_{\substack{s \to t \\ s \in \mathbb{Q}}}B_s(\omega)=B_t(\omega)
    \end{align*} if and only if for each $\eps>0$, there exists a $\delta>0$ such that for all $q \in \mathbb{Q}$, $|q-t|<\delta$ implies $\rho\mathopen{}\left(B_t(\omega),B_q(\omega)\right)\mathclose{}<\eps$. So \begin{align*}
        \mathcal{C}_t&=\bigcap_{\eps>0}\left(\bigcup_{\delta>0} \left\{\omega \in \Omega:\rho\mathopen{}\left(B_q(\omega),B_t(\omega)\right)\mathclose{}<\eps \quad \forall q \in \mathbb{Q}^+\cap \left(t-\delta,t+\delta\right)\right\}\right).
    \end{align*} For each $N \in \mathbb{N}$, we can write this as \begin{align*}
        \mathcal{C}_t&=\bigcap_{k=N}^\infty\left(\bigcup_{m=N}^\infty \left\{\omega \in \Omega:\rho\mathopen{}\left(B_q(\omega),B_t(\omega)\right)\mathclose{}<\frac1{k} \quad \forall q \in \mathbb{Q}^+\cap \left(t-\frac1{m},t+\frac1{m}\right)\right\}\right) \\
        &=\bigcap_{k=N}^\infty\left(\bigcup_{m=N}^\infty \left(\bigcap_{q \in \mathbb{Q}^+\cap \left(t-\frac1{m},t+\frac1{m}\right)}\left\{\omega \in \Omega:\widetilde{\rho}\mathopen{}\left(B_q(\omega),B_t(\omega)\right)\mathclose{}<\frac1{k} \right\}\right)\right),
    \end{align*}since from the definition of $\widetilde{\rho}$ in \eqref{widehatrhodef}, $\rho\le 1\iff \widetilde{\rho}=\rho$. $\widetilde{\rho}$ is continuous, so for each $q \in \mathbb{Q}^+\cap \left(t-\frac1{m},t+\frac1{m}\right)$, $k \in \mathbb{N}$, the set \begin{align*}
        \left\{\omega \in \Omega:\widetilde{\rho}\mathopen{}\left(B_q(\omega),B_t(\omega)\right)\mathclose{}<\frac1{k} \right\}
    \end{align*} is $\mathcal{F}_{t+\frac1{N}}$-measurable ($m\ge N$), and since $\mathbb{Q}^+$ is countable,  this implies $\mathcal{C}_t \in \mathcal{F}_{t+\frac1{N}}$, for each $N \in \mathbb{N}$. So $\mathcal{C}_t \in \mathcal{F}_t^+$. \newline

    We will now show $\mathbb{P}\bigl(\mathcal{C}_t\bigr)=1$. To do this it is easier to work with the complement. We can write it as \begin{align*}
        \Omega\setminus\mathcal{C}_t&=\bigcup_{k=1}^\infty \left\{\omega \in\Omega: \limsup_{\substack{q \to t \\ q \in \mathbb{Q}}}\rho\mathopen{}\left(B_q(\omega),B_t(\omega)\right)\mathclose{}>\frac1{k}\right\}.
    \end{align*} Observe that $\rho(x_1,x_2)>\frac1{k} \iff \widetilde{\rho}(x_1,x_2)>\frac1{k}$, so
    \begin{align}
        \Omega\setminus\mathcal{C}_t&=\bigcup_{k=1}^\infty \left\{\omega \in\Omega: \limsup_{\substack{q \to t \\ q \in \mathbb{Q}}}\widetilde{\rho}\mathopen{}\left(B_q(\omega),B_t(\omega)\right)\mathclose{}>\frac1{k}\right\}.
        \label{fave}
    \end{align}
   We start by noting that the set \begin{align*}
       \left\{\omega \in \Omega: \limsup_{\substack{q \to t \\ q \in \mathbb{Q}}}\widetilde{\rho}\mathopen{}\left(B_q(\omega),B_t(\omega)\right)\mathclose{}>\frac1{k}\right\}
   \end{align*} is $\mathcal{F}_\infty$-measurable, since we \begin{align*}
       &\left\{\omega \in \Omega: \limsup_{\substack{q \to t \\ q \in \mathbb{Q}}}\widetilde{\rho}\mathopen{}\left(B_q(\omega),B_t(\omega)\right)\mathclose{}>\frac1{k}\right\} \\
      &= \bigcap_{m=1}^\infty \left\{\omega \in \Omega:\exists q \in \mathbb{Q}\cap  \left(t-\frac1{m},t+\frac1{m}\right) \text{ such that } \widetilde{\rho}\mathopen{}\left(B_q(\omega),B_t(\omega)\right)\mathclose{}>\frac1{k}\right\} \\
      &=\bigcap_{m=1}^\infty\left(\bigcup_{q \in\mathbb{Q}\cap  \left(t-\frac1{m},t+\frac1{m}\right) }\left\{\omega \in \Omega: \widetilde{\rho}\mathopen{}\left(B_q(\omega),B_t(\omega)\right)>\frac1{k}\right\}\right),
   \end{align*} and for each $q \in \mathbb{Q}$, the set $\left\{\omega \in \Omega: \widetilde{\rho}\mathopen{}\left(B_q(\omega),B_t(\omega)\right)>\frac1{k}\right\}$ is $\mathcal{F}_\infty$-measurable. We will now show the set defined in \eqref{fave} is of measure $0$.

    For each $\eps \in \left(0,\frac1{2}\right)$ (with $t-\eps$ rational), let $\left\{\pi_j\right\}_{j=1}^\infty$ be a sequence of refined partitions (so $\pi_{j} \subset \pi_{j+1}$ for each $j$) of $\left[t-\eps,t+\eps\right]$ with $t$ included in each partition $\pi_j$.  So we can write \begin{align}
        \pi_j=\left\{\pi_{j,0}=t-\eps<\pi_{j,1}<\pi_{j,2}<\cdots<\pi_{j,h}=t<\cdots <\pi_{j,m-1}<\pi_{j,m}=t+\eps\right\}.
    \end{align} Moreover, let each point of $\pi_j$  be in $\mathbb{Q}^+\cup \{t\}$, and let $\bigcup_{j=1}^\infty \pi_j=\left(\mathbb{Q}^+\cup \{t\}\right)\cap [t-\eps,t+\eps]$. Since each $\pi_j$ is a partition of $[t-\eps,t+\eps]$, from Lemma \ref{goddid}, \begin{align*}
        \int_{\Omega}\LV\mathopen{}\left(\omega,\pi_j\right)\mathclose{}\mathbb{P}(\textup{d}\omega)\le 2K\eps,
    \end{align*} so \begin{align}
        \int_{\left\{\omega \in \Omega: \limsup_{\substack{q \to t \\ q \in \mathbb{Q}}}\widetilde{\rho}\mathopen{}\left(B_q\mathopen{}\left(\omega\right)\mathclose{},B_t\mathopen{}\left(\omega\right)\mathclose{}\right)\mathclose{}>\frac1{k}\right\}}\LV\mathopen{}\left(\omega,\pi_j\right)\mathclose{}\mathbb{P}(\textup{d}\omega)\le 2K\eps.
        \label{through}
    \end{align} 
    Let $\omega \in \Omega\setminus \mathcal{C}_t$ be such that \begin{align}
        \limsup_{\substack{q \to t \\ q \in \mathbb{Q}}}\widetilde{\rho}\mathopen{}\left(B_q\mathopen{}\left(\omega\right)\mathclose{},B_t\mathopen{}\left(\omega\right)\mathclose{}\right)\mathclose{}>\frac1{k}.
        \label{aleph9}
    \end{align} This implies for each $\eps>0$ we can find  $q_\eps \in \mathbb{Q}^+\cap (t-\eps,t+\eps)$ such that $\widetilde{\rho}\mathopen{}\left(B_{q_\eps}\mathopen{}\left(\omega\right)\mathclose{},B_t\mathopen{}\left(\omega\right)\mathclose{}\right)\mathclose{}>\frac1{k}$. Consider the partition (formed by) $\left\{q_\eps,t\right\}$. Since the partitions $\left\{\pi_j\right\}_{j=1}^\infty$ are getting finer and eventually they cover every point of $\left(\mathbb{Q}^+\cup \{t\}\right)\cap \left[t-\eps,t+\eps\right]$, for large enough $j$, $\left\{q_\eps,t\right\} \subset \pi_j$. So then \begin{align*}
        \liminf_{j \to \infty}\LV\mathopen{}\left(\omega,\pi_j\right)\mathclose{}\ge \LV\mathopen{}\left(\omega, \left\{q_\eps,t\right\}\right)\mathclose{}=\widetilde{\rho}\mathopen{}\left(B_{q_\eps}\mathopen{}\left(\omega\right)\mathclose{},B_t\mathopen{}\left(\omega\right)\mathclose{}\right)\mathclose{}> \frac1{k}.
    \end{align*} 
    This holds for any $\omega \in \Omega\setminus \mathcal{C}_t$ satisfying \eqref{aleph9}. Applying Fatou's Lemma, we get \begin{align*}
           & \liminf_{j \to \infty}\int_{\left\{\omega \in \Omega: \limsup_{\substack{q \to t \\ q \in \mathbb{Q}}}\widetilde{\rho}\mathopen{}\left(B_q\mathopen{}\left(\omega\right)\mathclose{},B_t\mathopen{}\left(\omega\right)\mathclose{}\right)\mathclose{}>\frac1{k}\right\}}\LV\mathopen{}\left(\omega,\pi_j\right)\mathclose{}\mathbb{P}(\textup{d}\omega) \\
           &\ge \int_{\left\{\omega \in \Omega: \limsup_{\substack{q \to t \\ q \in \mathbb{Q}}}\widetilde{\rho}\mathopen{}\left(B_q(\omega),B_t(\omega)\right)\mathclose{}>\frac1{k}\right\}}\liminf_{j \to \infty}\LV\mathopen{}\left(\omega,\pi_j\right)\mathbb{P}(\textup{d}\omega) \\
           &\ge \frac1{k}\int_{\left\{\omega \in \Omega: \limsup_{\substack{q \to t \\ q \in \mathbb{Q}}}\widetilde{\rho}\mathopen{}\left(B_q(\omega),B_t(\omega)\right)\mathclose{}>\frac1{k}\right\}}\mathbb{P}(\textup{d}\omega),
    \end{align*} and combining this with \eqref{through}, we get
    \begin{align*}
      2K\eps\ge \frac1{k}\mathbb{P}\mathopen{}\left(\left\{\omega \in \Omega: \limsup_{\substack{q \to t \\ q \in \mathbb{Q}}}\rho\mathopen{}\left(B_q(\omega),B_t(\omega)\right)\mathclose{}>\frac1{k}\right\}\right)\mathclose{},
    \end{align*}  for each $\eps>0$. This implies \begin{align*}
        \mathbb{P}\mathopen{}\left(\left\{\omega \in \Omega: \limsup_{\substack{q \to t \\ q \in \mathbb{Q}}}\widetilde{\rho}\mathopen{}\left(B_q(\omega),B_t(\omega)\right)\mathclose{}>\frac1{k}\right\}\right)\mathclose{}=0
    \end{align*} for each $k \in \mathbb{N}$. Combining this with \eqref{fave}, we get $\mathbb{P}\mathopen{}\left(\Omega\setminus \mathcal{C}_t\right)\mathclose{}=0$, as desired.
\end{proof}
Our final lemma is as follows.
\begin{lemma}
Let $\omega \in\Omega\setminus \bigcup_{T \in \mathbb{N}}\mathcal{S}_T$, where $\mathcal{S}_T$ is as defined in \eqref{mathcalSdef}. Then for each $t \in [0,\infty)$, the left and right-sided rational limits \begin{align*}
    \lim_{\substack{s \to t^- \\ s \in \mathbb{Q}^+}}B_s(\omega), \ \ \lim_{\substack{s \to t^+ \\ s \in \mathbb{Q^+}}}B_s(\omega)
\end{align*} exist.
    \label{lemmatouse}
\end{lemma}
\begin{proof}
    Let $t \in [0,\infty)$. We tackle the left-sided limit first. 

    Suppose for the sake of contradiction the limit $\lim_{\substack{s \to t^- \\ s \in \mathbb{Q}^+}}B_s(\omega)$ does not exist. Since $E$ is complete, this is equivalent to being Cauchy, so the limit exists if and only if for every $\eps \in (0,1)$, there exists a $\delta>0$ such that $\rho\mathopen{}\left(B_{q_1}(\omega),B_{q_2}(\omega)\right)\mathclose{}\le \eps$ for all $q_1,q_2 \in (t-\delta,t)$ (and since $\eps \in (0,1)$ we can replace $\rho$ with $\widetilde{\rho}$ in the previous statement). So if the limit does not exist, we can find an $\eps \in (0,1)$ such that for all $\delta>0$, there exists $q_1,q_2 \in (t-\delta,t)$ such that $\widetilde{\rho}\mathopen{}\left(B_{q_1}(\omega),B_{q_2}(\omega)\right)\mathclose{}>\eps$. If this happens, we can create an increasing sequence $q_1<q_2<q_3<...$ such that $q_j \to t^{-}$, and for each odd $j$, $\widetilde{\rho}\mathopen{}\left(B_{q_j}(\omega),B_{q_{j+1}}(\omega)\right)\mathclose{}>\eps$. Hence
 \begin{equation}
   \begin{split}
        &\left\{\omega \in \Omega: \lim_{\substack{s \to t^- \\ s \in \mathbb{Q}^+}}B_s(\omega)\text{ does not exist}\right\} \\
        &=\bigcup_{m=1}^\infty \left\{\omega \in \Omega:  \exists \text{ an increasing sequence } \left\{q_j\right\}_{j=1}^\infty \to t^{-}, \ \widetilde{\rho}\mathopen{}\left(B_{q_j}(\omega),B_{q_{j+1}}(\omega)\right)\mathclose{}>\frac1{m} \ \forall \text{ odd }j\right\}.
   \end{split}
        \label{ponishyou}
    \end{equation} Suppose $\omega \in \Omega$ is such that there exists an increasing sequence of rationals $q_1<q_2<q_3<\cdots$ such that $q_j \to t^-$, and for each odd $j$, $\widetilde{\rho}\mathopen{}\left(B_{q_j}(\omega),B_{q_{j+1}}(\omega)\right)\mathclose{}>\frac1{m}$. This then implies for each $N \in \mathbb{N}$, \begin{align*}
        \LV\mathopen{}\left(\omega, \left\{q_j\right\}_{j=1}^N\right)=\sum_{j=1}^{N-1}\widetilde{\rho}\mathopen{}\left(B_{q_j}(\omega),B_{q_{j+1}}(\omega)\right)\mathclose{}\ge\frac1{m}\left\lceil \frac{N-1}{2}\right\rceil
    \end{align*} Let $h \in \mathbb{N}$ be such that $h\ge t$. Since the points $\left\{q_j\right\}_{j=1}^\infty$ are rational and less than $t$, for large enough $k$, $\left\{q_j\right\}_{j=1}^N \subset \tau_k^h$, where the partition $\tau_k^h$ is as defined in Definition \ref{canonicalseqpartitionts}. Consequently, \begin{align*}
        \liminf_{k \to \infty}\LV\mathopen{}\left(\omega,\tau_k^h\right)\ge  \LV\mathopen{}\left(\omega, \left\{q_j\right\}_{j=1}^N\right)\ge\frac1{m}\left\lceil \frac{N-1}{2}\right\rceil
    \end{align*} for each $N \in \mathbb{N}$, so $\lim_{k \to \infty}\LV\mathopen{}\left(\omega,\tau_k^h\right)=\infty$. This implies \begin{align*}
        \left\{\omega \in \Omega:  \exists \text{ an increasing sequence } \left\{q_j\right\}_{j=1}^\infty \to t^{-}, \ \widetilde{\rho}\mathopen{}\left(B_{q_j}(\omega),B_{q_{j+1}}(\omega)\right)\mathclose{}>\frac1{m} \ \forall \text{ odd }j\right\} \subset \mathcal{S}_h,
    \end{align*} and so from \eqref{ponishyou}, \begin{align*}
        \left\{\omega \in \Omega: \lim_{\substack{s \to t^- \\ s \in \mathbb{Q}^+}}B_s(\omega)\text{ does not exist}\right\} \subset \mathcal{S}_h.
    \end{align*}
 This however is a contradiction, since by assumption $\omega \notin \bigcup_{T \in \mathbb{N}}\mathcal{S}_T$, so the limit \begin{align*}
        \lim_{\substack{s \to t^{-} \\ s\in \mathbb{Q}^+}}B_s(\omega)
    \end{align*} exists.
    \newline

Now for the right-sided limit. Suppose for the sake of contradiction the limit $\lim_{\substack{s \to t^+ \\ s \in \mathbb{Q}^+}}B_s(\omega)$ does not exist. Like with the left-sided limit, this means we can find an $\eps \in (0,1)$ such that for all $\delta>0$, there exists $q_1,q_2 \in (t,t+\delta)$ such that $\widetilde{\rho}\mathopen{}\left(B_{q_1}(\omega),B_{q_2}(\omega)\right)\mathclose{}>\eps$. If this happens, we can create a decreasing sequence $q_1>q_2>q_3>...$ such that $q_j \to t^{+}$, and for each odd $j$, $\widetilde{\rho}\mathopen{}\left(B_{q_j}(\omega),B_{q_{j+1}}(\omega)\right)\mathclose{}>\eps$. Hence
 \begin{equation}
   \begin{split}
        &\left\{\omega \in \Omega: \lim_{\substack{s \to t^+ \\ s \in \mathbb{Q}^+}}B_s(\omega)\text{ does not exist}\right\} \\
        &=\bigcup_{m=1}^\infty \left\{\omega \in \Omega:  \exists \text{ a decreasing sequence } \left\{q_j\right\}_{j=1}^\infty \to t^{+}, \ \widetilde{\rho}\mathopen{}\left(B_{q_j}(\omega),B_{q_{j+1}}(\omega)\right)\mathclose{}>\frac1{m} \ \forall \text{ odd }j\right\}.
   \end{split}
        \label{ponishyou2}
    \end{equation}
Suppose $\omega \in \Omega$ is such that there exists a decreasing sequence of rationals $q_1>q_2>q_3>\cdots$ such that $q_j \to t^+$, and for each odd $j$, $\widetilde{\rho}\mathopen{}\left(B_{q_j}(\omega),B_{q_{j+1}}(\omega)\right)\mathclose{}>\frac1{m}$. This then implies for each $N \in \mathbb{N}$, \begin{align*}
        \LV\mathopen{}\left(\omega, \left\{q_{N+1-j}\right\}_{j=1}^N\right)=\sum_{j=1}^{N-1}\widetilde{\rho}\mathopen{}\left(B_{q_{N+1-j}}(\omega),B_{q_{N-j}}(\omega)\right)\mathclose{}\ge \frac1{m}\left\lfloor \frac{N-1}{2}\right\rfloor.
    \end{align*} Let $h \in \mathbb{N}$ be such that $h\ge q_1$. Since the points $\left\{q_j\right\}_{j=1}^\infty$ are rational, for large enough $k$, $\left\{q_{N+1-j}\right\}_{j=1}^N \subset \tau_k^{h}$ (recall $\left\{q_{N+1-j}\right\}_{j=1}^N$ is decreasing), where the partition $\tau_k^{h}$ is as defined in Definition \ref{canonicalseqpartitionts}. Consequently, \begin{align*}
        \liminf_{k \to \infty}\LV\mathopen{}\left(\omega,\tau_k^h\right)\ge  \LV\mathopen{}\left(\omega, \left\{q_{N+1-j}\right\}_{j=1}^N\right)\ge \frac1{m}\left\lfloor \frac{N-1}{2}\right\rfloor
    \end{align*} for each $N \in \mathbb{N}$, so $\lim_{k \to \infty}\LV\mathopen{}\left(\omega,\tau_k^h\right)=\infty$. This implies \begin{align*}
        \left\{\omega \in \Omega:  \exists \text{ a decreasing sequence } \left\{q_j\right\}_{j=1}^\infty \to t^{+}, \ \widetilde{\rho}\mathopen{}\left(B_{q_j}(\omega),B_{q_{j+1}}(\omega)\right)\mathclose{}>\frac1{m} \ \forall \text{ odd }j\right\} \subset \mathcal{S}_h,
    \end{align*} and so from \eqref{ponishyou2}, \begin{align*}
        \left\{\omega \in \Omega: \lim_{\substack{s \to t^+ \\ s \in \mathbb{Q}^+}}B_s(\omega)\text{ does not exist}\right\} \subset \mathcal{S}_h.
    \end{align*}
 This however is a contradiction, since $\omega \notin \bigcup_{T \in \mathbb{N}}\mathcal{S}_T$, so the limit $\lim_{\substack{s \to t^{+} \\ s\in \mathbb{Q}^+}}B_s(\omega)$ exists.
This completes the proof of Lemma \ref{lemmatouse}.
\end{proof}

We are now ready to define $\left(\widetilde{B}_t\right)_{t\ge 0}$:
\subsection{Defining the process \texorpdfstring{$\left(\widetilde{B}_t\right)_{t\ge 0}$}{Lg}}\label{tildeBdef}
\noindent
\newline
For each $t>0$, we define $\widetilde{B}_t\colon \Omega \to E$ as follows.\newline
\noindent
\textbf{Case 1: $ \omega \in \bigcup_{T \in \mathbb{N}}\mathcal{S}_T$.}
In this case, pick any element $\widehat{e} \in E$ and set \begin{align*}
    \widetilde{B}_t(\omega)\equiv \widehat{e},
\end{align*} for all $t$.\newline

\noindent
\textbf{Case 2: $ \omega \notin \bigcup_{T \in \mathbb{N}}\mathcal{S}_T$.} From Lemma \ref{lemmatouse}, the limit $\lim_{\substack{s \to t^+ \\ s \in \mathbb{Q}^+}}B_s(\omega)$ exists, so we set \begin{align*}
    \widetilde{B}_t(\omega)=\lim_{\substack{s \to t^+ \\ s \in \mathbb{Q}^+}}B_s(\omega).
\end{align*}
\noindent
\newline

\noindent
\section{Modification of \texorpdfstring{$B$}{Lg} to c\`{a}dl\`{a}g paths}
We will now show that
\begin{proposition}
    The process $\left(\widetilde{B}_t\right)_{t\ge 0}$ is adapted to the filtration $\left(\widetilde{\mathcal{F}}_t\right)_{t\ge 0}$, and is a c\`{a}dl\`{a}g modification of $\left(B_t\right)_{t\ge 0}$ 
\label{levitating}
\end{proposition}
We do this in parts:
\begin{proof}[Proof that $\left(\widetilde{B}_t\right)_{t\ge 0}$ is adapted to $\left(\widetilde{\mathcal{F}}_t\right)_{t\ge 0}$]
Recall $\mathcal{E}$ is the Borel $\sigma$-algebra on $E$, and we assumed that $E$ is a metrizable locally compact topological space. We also assumed that $E$ is $\sigma$-compact. In particular $E$ is Hausdorff $\sigma$-compact. In that case, it turns out that $\mathcal{E}$ is generated by the compact sets in $E$, so
it suffices to show 
 $\widetilde{B}_t^{-1}(X) \in \widetilde{\mathcal{F}}_t$ for each compact set $X$.
 To that end, let $X \subset E$ be compact.
 We will show that $\widetilde{B}_t^{-1}(X) \in \widetilde{\mathcal{F}}_t$. By definition $ \widetilde{B}_t^{-1}(X)=\left\{\omega \in \Omega:\widetilde{B}_t(\omega) \in X\right\}$, and by considering the definition of $\widetilde{B}_t$ in Subsection \ref{tildeBdef}, we can write this as \begin{equation}
\begin{split}
    \widetilde{B}_t^{-1}(X)&=\left\{\omega \in \bigcup_{T \in \mathbb{N}}\mathcal{S}_T: \widetilde{B}_t(\omega) \in X\right\}\cup \left\{\omega \in \Omega\setminus \bigcup_{T \in \mathbb{N}}\mathcal{S}_T: \widetilde{B}_t(\omega) \in X\right\}  \\
    &=\left\{\omega \in \bigcup_{T \in \mathbb{N}}\mathcal{S}_T: \widetilde{B}_t(\omega) \in X\right\}\cup \left\{\omega \in \Omega\setminus \bigcup_{T \in \mathbb{N}}\mathcal{S}_T: \lim_{\substack{s \to t^+ \\ s \in \mathbb{Q}^+}}B_s(\omega) \in X\right\} .
\end{split}
    \label{BtinvA}
\end{equation} First note if $\omega \in \bigcup_{T \in \mathbb{N}}\mathcal{S}_T$, then $\widetilde{B}_t(\omega)=\widehat{e}$, so \begin{align*}
     \left\{\omega \in \bigcup_{T \in \mathbb{N}}\mathcal{S}_T: \widetilde{B}_t(\omega) \in X\right\}=\begin{cases}
         \bigcup_{T \in \mathbb{N}}\mathcal{S}_T & \text{ if }\widehat{e} \in X, \\
         \emptyset & \text{ otherwise.}
     \end{cases}
\end{align*}
 Note that $\bigcup_{T \in \mathbb{N}}\mathcal{S}_T$ is $\mathcal{F}_\infty$-measurable, and $\mathbb{P}\left(\bigcup_{T \in \mathbb{N}}\mathcal{S}_T\right)=0$ from Lemma \ref{canon}, so $\bigcup_{T \in \mathbb{N}}\mathcal{S}_T \in \sigma(\mathcal{N})$. Either way,  \begin{align}
\left\{\omega \in \bigcup_{T \in \mathbb{N}}\mathcal{S}_T: \widetilde{B}_t(\omega) \in X\right\} \in \sigma(\mathcal{N}).
    \label{weather}
\end{align} From Lemma \ref{lemmatouse}, we know if $\omega \in \Omega\setminus \bigcup_{T \in \mathbb{N}}\mathcal{S}_T$, then the limit $\lim_{\substack{s \to t^+ \\ s \in \mathbb{Q}^+}}B_s(\omega)$ exists. Moreover, we can note that given that the limit $\lim_{\substack{s \to t^+ \\ s \in \mathbb{Q}^+}}B_s(\omega)$ exists and $X$ is compact, this limit is in $X$ if and only if 
for any $\eps>0$, there exists some $\delta>0$ such that $\dist\mathopen{}\left(B_q(\omega),X\right)\mathclose{}<\eps$ for all $q \in \mathbb{Q}\cap (t,t+\delta)$, where \begin{align*}
    \dist\mathopen{}\left(s,X\right)\mathclose{}\coloneqq \inf_{x \in X}\rho(x,s).
\end{align*}
 In other words, \begin{equation}
\begin{split}
     &\left\{\omega \in \Omega\setminus \bigcup_{T \in \mathbb{N}}\mathcal{S}_T:\lim_{\substack{s \to t^+ \\ s \in \mathbb{Q}^+}}B_s(\omega) \in X\right\} \\
    &=\left(\Omega\setminus \bigcup_{T \in \mathbb{N}}\mathcal{S}_T\right)\cap \left\{\omega \in \Omega: \forall\eps>0, \ \exists \delta>0 \text{ such that } \forall q \in \mathbb{Q}\cap (t,t+\delta), \ \dist(B_q(\omega),X)<\eps\right\} \\
    &=\left(\Omega\setminus \bigcup_{T \in \mathbb{N}}\mathcal{S}_T\right)\cap \bigcap_{k=N}^\infty\left(\bigcup_{m=N}^\infty\bigcap_{q \in \mathbb{Q}\cap \left(t,t+\frac1{m}\right)}\left\{\omega \in \Omega: \dist\mathopen{}\left(B_q(\omega),X\right)\mathclose{}<\frac1{k}\right\}\right),
\end{split}
    \label{away}
\end{equation} for each $N \in \mathbb{N}$. Note $s \mapsto \dist(s,X)$ is continuous, so this implies for each $q$, the set $\left\{\omega \in \Omega: \dist\mathopen{}\left(B_q(\omega),X\right)\mathclose{}<\frac1{k}\right\}$ is $\mathcal{F}_q$-measurable, so \begin{align*}
    \bigcap_{k=N}^\infty\left(\bigcup_{m=N}^\infty\bigcap_{q \in \mathbb{Q}\cap \left(t,t+\frac1{m}\right)}\left\{\omega \in \Omega: \dist\mathopen{}\left(B_q(\omega),X\right)\mathclose{}<\frac1{k}\right\}\right) \in \mathcal{F}_{t+\frac1{N}} 
\end{align*} for each $N \in \mathbb{N}$. From \eqref{away}, this implies \begin{align*}
    \left\{\omega \in \Omega: \forall\eps>0, \ \exists \delta>0 \text{ such that } \forall q \in \mathbb{Q}\cap (t,t+\delta), \ \dist(B_q(\omega),X)<\eps\right\} \in \mathcal{F}_t^+,
\end{align*} and so \begin{align*}
    \left\{\omega \in \Omega\setminus \bigcup_{T \in \mathbb{N}}\mathcal{S}_T:\lim_{\substack{s \to t^+ \\ s \in \mathbb{Q}^+}}B_s(\omega) \in X\right\}=\sigma\mathopen{}\left(\mathcal{F}_t^+\cup \mathcal{N}\right)\mathclose{}=\widetilde{\mathcal{F}}_t.
    \end{align*}
Combining this result with \eqref{weather} and comparing to \eqref{BtinvA}, this shows $\widetilde{B}_t^{-1}(X) \in \widetilde{\mathcal{F}}_t$, for any compact set $X \subset E$. Hence $\widetilde{B}_t$ is $\widetilde{\mathcal{F}}_t$-measurable, and so $\left(\widetilde{B}_t\right)_{t\ge 0}$ is adapted to the filtration $\left(\widetilde{\mathcal{F}}_t\right)_{t\ge 0}$, as desired.   
\end{proof}
\noindent
\newline

We will now show $\left(\widetilde{B}_t\right)_{t\ge 0}$ is a modification of $\left(B_t\right)_{t\ge 0}$:
\begin{proof}[Proof that $\left(\widetilde{B}_t\right)_{t\ge 0}$ is a modification of $\left(B_t\right)_{t\ge 0}$]
Note since $\left(\widetilde{B}_t\right)_{t\ge 0}$ is adapted to $\left(\widetilde{F}_{t\in [0,\infty]}\right)$, the set $ \left\{\omega \in \Omega:B_t(\omega)=\widetilde{B}_t(\omega)\right\}$  is $\widetilde{\mathcal{F}}_t$-measurable. From Lemma \ref{Yanni}, if $\omega \in \mathcal{C}_t$, then \begin{align*}
    \lim_{\substack{s \to t \\ s\in \mathbb{Q}}}B_s(\omega)=B_t(\omega),
\end{align*} and from the definition of $\widetilde{B}_t$ in Subsection \ref{tildeBdef}, we know if $\omega \in \Omega\setminus \bigcup_{T \in \mathbb{N}}\mathcal{S}_T$, then \begin{align*}
    \widetilde{B}_t(\omega)=\lim_{\substack{s \to t^+ \\ s \in \mathbb{Q}^+}}B_s(\omega).
\end{align*} So if $\omega \in \mathcal{C}_t\cap \left(\Omega\setminus \bigcup_{T \in \mathbb{N}}\mathcal{S}_T\right)$, then $B_t(\omega)=\widetilde{B}_t(\omega)$. Hence \begin{align}
   \mathcal{C}_t\cap \left(\Omega\setminus \bigcup_{T \in \mathbb{N}}\mathcal{S}_T\right)\subset \left\{\omega \in \Omega:B_t(\omega)=\widetilde{B}_t(\omega)\right\}.
   \label{jcole}
\end{align} From Lemma \ref{canon}, we can deduce $\mathbb{P}\left(\bigcup_{T \in \mathbb{N}}\mathcal{S}_T\right)=0$, and from Lemma \ref{Yanni}, $\mathbb{P}\left(\mathcal{C}_t\right)=1$, so \begin{align*}
    \mathbb{P}\left(\mathcal{C}_t\cap \left(\Omega\setminus \bigcup_{T \in \mathbb{N}}\mathcal{S}_T\right)\right)=1.
\end{align*} This combined with \eqref{jcole} implies $\mathbb{P}\left( \left\{\omega \in \Omega:B_t(\omega)=\widetilde{B}_t(\omega)\right\}\right)=1$, as desired. 
\end{proof}
\noindent
\newline

Finally, we will now show that $\left(\widetilde{B}_t\right)_{t\ge 0}$ is c\`{a}dl\`{a}g:
\begin{proof}[Proof that $\left(\widetilde{B}_t\right)_{t\ge 0}$ is c\`{a}dl\`{a}g]
    First note from the definition of $\left(\widetilde{B}_{t\ge 0}\right)$ in Subsection \ref{tildeBdef} that if $\omega \in \bigcup_{T \in \mathbb{N}}\mathcal{S}_T$, then $t \mapsto \widetilde{B}_t(\omega)$ is constant, so it is certainly c\`{a}dl\`{a}g. Now we consider $\omega \notin \bigcup_{T \in \mathbb{N}}\mathcal{S}_T$.\newline

    Fix $\omega \in \Omega\setminus \bigcup_{T \in \mathbb{N}}\mathcal{S}_T$, and consider arbitrary $t_0 \in [0,\infty)$. We start by showing $t\mapsto \widetilde{B}_t(\omega)$ is right-continuous at $t_0$. We use the $\eps$-$\delta$ definition. Consider any $\eps>0$. From Lemma \ref{lemmatouse}, and the definition of $\widetilde{B}_t(\omega)$, we have \begin{align*}
       \widetilde{B}_{t_0}(\omega)= \lim_{\substack{s\to t_0^+ \\ s \in \mathbb{Q}^+}}B_s(\omega).
    \end{align*} To that end end, let $\delta>0$ be such that \begin{align}
        \rho\mathopen{}\left(B_q(\omega),\widetilde{B}_{t_0}(\omega)\right)\mathclose{}<\frac{\eps}{2} \quad \forall q \in \mathbb{Q}\cap \left(t_0,t_0+\delta\right).
        \label{jermaine}
    \end{align} Now let  $t \in \left(t_0,t_0+\delta\right)$ be arbitrary. Since $\lim_{\substack{s \to t^+ \\ s \in \mathbb{Q}}}B_s(\omega)=\widetilde{B}_t(\omega)$, there exists some  rational number $q_t \in \mathbb{Q}\cap \left(t,t_0+\delta\right)$ such that $\rho\mathopen{}\left(B_{q_t}(\omega), \widetilde{B}_t(\omega)\right)\mathclose{}<\frac{\eps}{2}$. Note $\left(t,t_0+\delta\right) \subset \left(t_0,t_0+\delta\right)$, so \eqref{jermaine} implies $\rho\mathopen{}\left(B_{q_t}(\omega),\widetilde{B}_{t_0}(\omega)\right)<\frac{\eps}{2}$, so by the triangle inequality, \begin{align*}
        \rho\left(\widetilde{B}_t(\omega),\widetilde{B}_{t_0}(\omega)\right)\le \rho\mathopen{}\left(\widetilde{B}_{t_0}(\omega),B_{q_t}(\omega)\right)+\rho\mathopen{}\left(B_{q_t}(\omega), \widetilde{B}_t(\omega)\right)\mathclose{}<\frac{\eps}{2}+\frac{\eps}{2}=\eps.
    \end{align*} $t \in \left(t_0,t_0+\delta\right)$ was arbitrary, so this shows that \begin{align*}
        \rho\left(\widetilde{B}_t(\omega),\widetilde{B}_{t_0}(\omega)\right)<\eps \quad \forall t \in \left(t_0,t_0+\delta\right).
    \end{align*} There exists a corresponding $\delta>0$ for each $\eps>0$, and so this shows that \begin{align*}
        \lim_{t \to t_0^+}\widetilde{B}_t(\omega)=\widetilde{B}_{t_0}(\omega).
    \end{align*}
\newline

    Now to show the left-limit exists. From Lemma \ref{lemmatouse}, we know the left-sided limit \begin{align*}
        \ell \coloneqq \lim_{\substack{s \to t_0^- \\ s \in \mathbb{Q}^+}}B_s(\omega)
    \end{align*} exists. To that end, we will show that $\lim_{t \to t_0^-}\widetilde{B}_t(\omega)=\ell$. We use the $\eps$-$\delta$ definition of continuity.
    Consider any $\eps>0$. We know there exists $\delta>0$ such that \begin{align}
       \rho\mathopen{}\left(B_q(\omega),\ell\right)\mathclose{}<\frac{\eps}{2} \quad \forall q \in \mathbb{Q}\cap \left(t_0-\delta,t_0\right).
       \label{pieces}
    \end{align} Now, let $t \in \left(t_0-\delta,t_0\right)$ be arbitrary. Since $\lim_{\substack{s \to t^{-} \\ s \in \mathbb{Q}}}B_s(\omega)=\widetilde{B}_t(\omega)$, there exists a rational $q_t \in \left(t_0-\delta,t\right)$ such that $\rho\mathopen{}\left(B_{q_t}(\omega), \widetilde{B}_t(\omega)\right)\mathclose{}<\frac{\eps}{2}$. Note $\left(t_0-\delta,t\right)\subset \left(t_0-\delta,t_0\right)$, so \eqref{pieces} implies $\rho\mathopen{}\left(B_{q_t}(\omega),\ell\right)<\frac{\eps}{2}$. So by the triangle inequality, \begin{align*}
        \rho\mathopen{}\left(\widetilde{B}_t(\omega),\ell\right)\mathclose{}\le \rho\mathopen{}\left(\ell,B_{q_t}(\omega)\right)\mathclose{}+\rho\mathopen{}\left(B_{q_t}(\omega),\widetilde{B}_t(\omega)\right)\mathclose{}<\frac{\eps}{2}+\frac{\eps}{2}=\eps.
    \end{align*} $t \in \left(t_0-\delta,t_0\right)$ was arbitrary, so this shows that \begin{align*}
        \rho\mathopen{}\left(\widetilde{B}_t(\omega),\ell\right)\mathclose{}<\eps \quad \forall t \in \left(t_0-\delta,t_0\right).
\end{align*} There exists a corresponding $\delta>0$ for each $\eps>0$, and so this shows that \begin{align*}
    \lim_{t \to t_0^-}\widetilde{B}_t(\omega)=\ell. 
\end{align*} So the left limit exists.

    This holds for all $t_0\ge 0$, and so $t\mapsto \widetilde{B}_t(\omega)$ is c\`{a}dl\`{a}g, as desired.
\end{proof}
These three steps complete the proof of Proposition \ref{levitating}.\newline

\noindent
We now show $\left(\widetilde{B}_{t}\right)_{t\ge 0}$ is a Markov process with semi-group $\left(Q_{t}\right)_{t\ge 0}$ and filtration $\left(\widetilde{\mathcal{F}}_t\right)_{t\ge 0}$:
\begin{proposition}
    For all $t,s\ge 0$, and bounded, continuous functions $f\colon E \to \mathbb{R}$, we have  \begin{enumerate}[i]
     \item \label{jets}$\mathbb{E}\left[f\mathopen{}\left(\widetilde{B}_{t+s}\right)\mathclose{}\Big\vert \widetilde{B}_{s}\right]=Q_{t}f\mathopen{}\left(\widetilde{B}_s\right)\mathclose{}$, and
     
     \item \label{jets2} $ \mathbb{E}\left[f\mathopen{}\left(\widetilde{B}_{t+s}\right)\mathclose{} \Big\lvert \widetilde{\mathcal{F}}_s\right]=Q_{t}f\mathopen{}\left(\widetilde{B}_s\right)\mathclose{}$.
 \end{enumerate}
\end{proposition}
\begin{proof}[Proof of \ref{jets}]
    Let $G \in \sigma\mathopen{}\left(\widetilde{B}_s\right)\mathclose{}$, so $G=\widetilde{B}_s^{-1}(X)$, for some $X \in \mathcal{E}$. We will show for any such $X$, we have \begin{align*}
        \mathbb{E}\left[f\mathopen{}\left(\widetilde{B}_{t+s}\right)\mathclose{}\mathbbm{1}_{\widetilde{B}_s^{-1}(X)}\right]=\mathbb{E}\left[Q_{t}f\mathopen{}\left(\widetilde{B}_s\right)\mathclose{}\mathbbm{1}_{\widetilde{B}_s^{-1}(X)}\right].
    \end{align*} By definition \begin{align*}
        \mathbb{E}\left[f\mathopen{}\left(\widetilde{B}_{t+s}\right)\mathclose{}\mathbbm{1}_{\widetilde{B}_s^{-1}(X)}\right]=\int_{\Omega}f\mathopen{}\left(\widetilde{B}_{t+s}(\omega)\right)\mathclose{}\mathbbm{1}_{\widetilde{B}_s^{-1}(X)}(\omega)\mathbb{P}(\textup{d}\omega),
    \end{align*} but since $B_{t+s}=\widetilde{B}_{t+s}$ holds $\mathbb{P}$-almost everywhere, this means\begin{align*}
        \mathbb{E}\left[f\mathopen{}\left(\widetilde{B}_{t+s}\right)\mathclose{}\mathbbm{1}_{\widetilde{B}_s^{-1}(X)}\right]=\int_{\Omega}f\left(B_{t+s}(\omega)\right)\mathbbm{1}_{\widetilde{B}_s^{-1}(X)}(\omega)\mathbb{P}(\textup{d}\omega).
    \end{align*} Similarly, since $\widetilde{B}_s=B_s$ holds $\mathbb{P}$-almost everywhere, we have $\mathbbm{1}_{\widetilde{B}_s^{-1}(X)}(\omega)=\mathbbm{1}_{\widetilde{B}_s^{-1}(X)}(\omega)$ for $\mathbb{P}$-almost every $\omega$: Indeed, suppose $\omega$ is such that $\widetilde{B}_s(\omega)=B_s(\omega)$. Then $B_s(\omega) \in X$ if and only if $\widetilde{B}_s(\omega) \in X$, so $\mathbbm{1}_{\widetilde{B}_s^{-1}(X)}(\omega)=\mathbbm{1}_{\widetilde{B}_s^{-1}(X)}(\omega)$, as desired. Hence $\mathbbm{1}_{\widetilde{B}_s^{-1}(X)}(\omega)=\mathbbm{1}_{\widetilde{B}_s^{-1}(X)}(\omega)$ for $\mathbb{P}$-almost every $\omega$. Consequently, \begin{align*}
        \mathbb{E}\left[f\mathopen{}\left(\widetilde{B}_{t+s}\right)\mathclose{}\mathbbm{1}_{\widetilde{B}_s^{-1}(X)}\right]=\int_{\Omega}f\mathopen{}\left(B_{t+s}(\omega)\right)\mathclose{}\mathbbm{1}_{B_s^{-1}(X)}(\omega)\mathbb{P}(\textup{d}\omega),
    \end{align*} so \begin{align}
        \mathbb{E}\left[f\mathopen{}\left(\widetilde{B}_{t+s}\right)\mathclose{}\mathbbm{1}_{\widetilde{B}_s^{-1}(X)}\right]=\mathbb{E}\left[f\mathopen{}\left(B_{t+s}\right)\mathclose{}\mathbbm{1}_{B_s^{-1}(X)}\right].
        \label{lowlife1}
    \end{align} This holds for any $f\colon E \to \mathbb{R}$ and $t,s\ge 0$, so we may deduce \begin{align}
        \mathbb{E}\left[Q_{t}f\mathopen{}\left(\widetilde{B}_{s}\right)\mathclose{}\mathbbm{1}_{\widetilde{B}_s^{-1}(X)}\right]=\mathbb{E}\left[Q_{t}f\mathopen{}\left(B_s\right)\mathclose{}\mathbbm{1}_{B_s^{-1}(X)}\right].
        \label{lowlife2}
    \end{align} From the Markov property, we know \begin{align*}
        \mathbb{E}\left[f(B_{t+s})|B_s\right]=Q_{t}f(B_s),
    \end{align*} so this implies \begin{align*}
        \mathbb{E}\left[f(B_{t+s})\mathbbm{1}_{B_s^{-1}(X)}\right]=\mathbb{E}\left[Q_{t}f(B_s)\mathbbm{1}_{B_s^{-1}(X)}\right]
    \end{align*} for any $X \in \mathcal{E}$. Combining this with \eqref{lowlife1} and \eqref{lowlife2} then gives \begin{align*}
        \mathbb{E}\left[f\mathopen{}\left(\widetilde{B}_{t+s}\right)\mathclose{}\mathbbm{1}_{\widetilde{B}_s^{-1}(X)}\right]=\mathbb{E}\left[Q_{t}f\mathopen{}\left(\widetilde{B}_s\right)\mathclose{}\mathbbm{1}_{\widetilde{B}_s^{-1}(X)}\right]
    \end{align*} for any $X \in \mathcal{E}$, which implies \begin{align*}
        \mathbb{E}\left[f\mathopen{}\left(\widetilde{B}_{t+s}\right)\mathclose \Big\lvert \widetilde{B}_s\right]=Q_{t}f\mathopen{}\left(\widetilde{B}_s\right)\mathclose{},
    \end{align*}  as desired.\newline
 \end{proof}
Finally, we prove statement \ref{jets2}.
\begin{proof}[Proof of \ref{jets2}]
 Recall that $B_t=\widetilde{B}_t$ $\mathbb{P}$-almost everywhere for all $t\ge 0$, so from the Markov Property, we may deduce \begin{align}
     \mathbb{E}\left[f\mathopen{}\left(\widetilde{B}_{t+r}\right)\mathclose{}\Big\lvert\mathcal{F}_r\right]=Q_{t}f\mathopen{}\left(\widetilde{B}_r\right)\mathclose{} \quad \text{for all } r,t\ge 0.
     \label{dinero}
 \end{align} Recall $\widetilde{\mathcal{F}}_s=\sigma\mathopen{}\left(\mathcal{F}_s^+\cup \mathcal{N}\right)$, so we want to show \begin{align*}
     \mathbb{E}\left[f\mathopen{}\left(\widetilde{B}_{t+s}\right)\mathclose{}\Big\lvert\sigma\bigl(\mathcal{F}_s^+\cup \mathcal{N}\bigr)\right]=Q_{t}f\mathopen{}\left(\widetilde{B}_s\right)\mathclose{}
 \end{align*} for all $s,t\ge 0$ to prove statement \ref{jets2}. Recall $\widetilde{B}_s$ is $\sigma\mathopen{}\left(\mathcal{F}_s^+\cup \mathcal{N}\right)\mathclose{}$-measurable, so $Q_{t}f\mathopen{}\left(\widetilde{B}_s\right)\mathclose{}$ is $\sigma\mathopen{}\left(\mathcal{F}_s^+\cup \mathcal{N}\right)\mathclose{}$-measurable, so it remains to show that \begin{align*}
     \mathbb{E}\left[f\mathopen{}\left(\widetilde{B}_{t+s}\right)\mathclose{}\mathbbm{1}_G\right]=\mathbb{E}\left[Q_{t}f\mathopen{}\left(\widetilde{B}_s\right)\mathclose{}\mathbbm{1}_G\right]
 \end{align*} for all $G \in \sigma\mathopen{}\left(\mathcal{F}_s^+\cup \mathcal{N}\right)\mathclose{}$. From Lemma \ref{coachella}, we may deduce that for any $G \in \sigma\mathopen{}\left(\mathcal{F}_s^+\cup \mathcal{N}\right)\mathclose{}$, we can find $\widetilde{G} \in \mathcal{F}_s^+$ such that $\mathbbm{1}_G-\mathbbm{1}_{\widetilde{G}}$ is a null-function, so it suffices to show that  \begin{align*}
     \mathbb{E}\left[f\mathopen{}\left(\widetilde{B}_{t+s}\right)\mathclose{}\mathbbm{1}_G\right]=\mathbb{E}\left[Q_{t}f\mathopen{}\left(\widetilde{B}_s\right)\mathclose{}\mathbbm{1}_G\right]
 \end{align*} for all $G \in \mathcal{F}_s^+$. Let $G \in \mathcal{F}_s^+$, so $G \in \mathcal{F}_r$ for all $r>s$. In that case, \eqref{dinero} implies \begin{align}
     \mathbb{E}\left[f\mathopen{}\left(\widetilde{B}_{t+r}\right)\mathclose{}\mathbbm{1}_G\right]=\mathbb{E}\left[Q_{t}f\mathopen{}\left(\widetilde{B}_r\right)\mathclose{}\mathbbm{1}_G\right] \quad \forall r>s.
     \label{madonna}
 \end{align} Since $\left(\widetilde{B}_t\right)_{t\ge 0}$ is c\`{a}dl\`{a}g and $f$ is continuous, for each $\omega \in \Omega$, \begin{align*}
     \lim_{r \to s^+}Q_tf\mathopen{}\left(\widetilde{B}_r(\omega)\right)\mathclose{}\mathbbm{1}_G(\omega)&=Q_tf\mathopen{}\left(\widetilde{B}_s(\omega)\right)\mathbbm{1}_G(\omega), \\
     \lim_{r \to s^+}f\mathopen{}\left(\widetilde{B}_{t+r}(\omega)\right)\mathclose{}\mathbbm{1}_G(\omega)&=f\mathopen{}\left(\widetilde{B}_{t+s}(\omega)\right)\mathclose{}\mathbbm{1}_G(\omega).
 \end{align*} Furthermore, as each $Q_t$ is a contraction of $C(E)$,  \begin{align*}
     \left\lVert f\mathopen{}\left(\widetilde{B}_{t+r}\right)\mathclose{}\mathbbm{1}_G\right\rVert_{C(E)}, \left\lVert Q_{t}f\mathopen{}\left(\widetilde{B}_r\right)\mathclose{}\mathbbm{1}_G\right\rVert_{C(E)}\le \Vert f\rVert_{C(E)}.
 \end{align*}
 So by dominated convergence, we can let $r \to s^+$ in \eqref{madonna} to get \begin{align*}
     \mathbb{E}\left[f\mathopen{}\left(\widetilde{B}_{t+s}\right)\mathclose{}\mathbbm{1}_G\right]=\mathbb{E}\left[Q_{t}f\mathopen{}\left(\widetilde{B}_s\right)\mathclose{}\mathbbm{1}_G\right],
 \end{align*}
for all $G \in \mathcal{F}_s^+$. This \begin{align*}
     \mathbb{E}\left[f\mathopen{}\left(\widetilde{B}_{t+s}\right)\mathclose{} \Big\lvert \widetilde{\mathcal{F}}_s\right]=Q_{t}f\mathopen{}\left(\widetilde{B}_s\right)\mathclose{},
 \end{align*} for all $s,t\ge 0$. This completes the proof of statement \ref{jets2}.
\end{proof}

 { \small 
	\bibliographystyle{plain}
	\bibliography{bib.bib} }

\end{document}